\documentclass[letterpaper,11pt]{amsart}
\usepackage{amsmath,amssymb,amsthm,pinlabel,tikz,hyperref,mathrsfs,color}
\usepackage{verbatim}
\usepackage{bm}
\newcommand{\nc}{\newcommand}
\nc{\dmo}{\DeclareMathOperator}
\dmo{\ra}{\rightarrow}
\dmo{\N}{\mathbb{N}}
\dmo{\Z}{\mathbb{Z}}
\dmo{\Q}{\mathbb{Q}}
\dmo{\R}{\mathbb{R}}
\dmo{\C}{\mathcal{C}}
\dmo{\AC}{\mathcal{AC}}
\dmo{\Mod}{Mod}
\dmo{\PMod}{PMod}
\dmo{\B}{B}
\dmo{\PB}{PB}
\dmo{\I}{\mathcal{I}}
\dmo{\el}{\ell_{\C}}
\dmo{\NN}{\mathcal{N}} 
\dmo{\rk}{rk}
\newcommand{\vol}{{\rm vol}}

\usetikzlibrary{decorations.markings}
\tikzset{->-/.style={decoration={
  markings,
  mark=at position #1 with {\arrow[scale=3]{>}}},postaction={decorate}}}
  
\nc{\nt}{\newtheorem}

\nt{theorem}{Theorem}
\nt{lemma}{Lemma}

\newtheorem{thm}{{\bf Theorem}}[section]
\newtheorem{lem}[thm]{{\bf Lemma}}
\newtheorem{cor}[thm]{{\bf Corollary}}
\newtheorem{prop}[thm]{{\bf Proposition}}

\newtheorem{remark}[thm]{Remark}
\newtheorem{ex}[thm]{Example}

\newtheorem{ques}[thm]{Question}

\newtheorem{conj}[thm]{Conjecture}

\numberwithin{equation}{section}
\nt*{theorem_nonumber}{Theorem}

\begin{document} 

\dedicatory{Dedicated to the Memory of Toshie Takata} 

\title[Braids and fibered $2$-fold branched covers of $3$-manifolds] 
{Braids, entropies and fibered $2$-fold branched covers of  $3$-manifolds}

\author[S. Hirose]{%
Susumu Hirose
}
\address{%
Department of Mathematics,  
Faculty of Science and Technology, 
Tokyo University of Science, 
Noda, Chiba, 278-8510, Japan}
\email{%
hirose\_susumu@ma.noda.tus.ac.jp
}

\author[E. Kin]{%
    Eiko Kin
}
\address{%
      Center for Education in Liberal Arts and Sciences, Osaka University, Toyonaka, Osaka 560-0043, Japan
}
\email{%
        kin.eiko.celas@osaka-u.ac.jp
}



\date{	
October 1, 2022}

\thanks{Hirose's research was partially supported by JSPS KAKENHI Grant Numbers JP16K05156, JP20K03618. 
Kin's research was partially supported by JSPS KAKENHI  Grant Numbers JP21K03247, JP22H01125.} 
	
\begin{abstract} 
It is proved by Sakuma and Brooks that 
any closed orientable $3$-manifold with a Heegaard splitting of genus $g$   
admits a $2$-fold branched cover  that is a hyperbolic $3$-manifold and 
a genus $g$ surface bundle over the circle. 
This paper concerns entropy of pseudo-Anosov monodromies for hyperbolic fibered $3$-manifolds. 
We prove that there exist infinitely many closed orientable $3$-manifolds $M$ such that 
the minimal entropy over all hyperbolic, genus $g$ surface bundles over the circle 
as $2$-fold branched covers of the $3$-manifold $M$ is comparable to $1/g$. 
\end{abstract}
\maketitle

\section{Introduction}
\label{section_Introduction}

Let $M$ be a closed orientable $3$-manifold 
which admits a genus $g$ Heegaard splitting. 
Sakuma  \cite{Sakuma81} proved that 
there exists a $2$-fold  branched cover $\widetilde{M}$ of $M$  
such that $\widetilde{M}$ is a genus $g$  surface bundle over the circle $S^1$. 
It is proved by Brooks \cite{Brooks85} that  
 the  branched cover $\widetilde{M}$ of $M$ can  be chosen to be hyperbolic if $g \ge 2$.

To state our results of this paper, 
let $\Sigma = \Sigma_{g,p}$ be an orientable, 
connected surface of genus $g$ with $p$ punctures, possibly $p= 0$. 
We set $\Sigma_g= \Sigma_ {g,0}$ 
for a closed orientable surface of genus $g$. 
The mapping class group $\mathrm{MCG}(\Sigma)$  is 
the group of isotopy classes of orientation-preserving 
self-homeomorphisms on $\Sigma$ which preserve the punctures setwise. 
%

By the Nielsen-Thurston classification \cite{Thurston88,FarbMargalit12}, 
an element  in $\mathrm{MCG}(\Sigma)$ is one of the following types: 
periodic, reducible, pseudo-Anosov. 
If an element in $\mathrm{MCG}(\Sigma)$ is neither periodic nor reducible, then 
it is pseudo-Anosov. 
For a mapping class $\phi= [f] \in \mathrm{MCG}(\Sigma)$, 
the {\it mapping torus}  $ T_{\phi}$ of $\phi$ is defined by 
$$
T_{\phi} = \Sigma \times \mathbb{R} / \sim, 
$$
where $(x,t) \sim (f(x), t+1)$ for $x \in \Sigma$ and $t \in \mathbb{R}$. 
We call $\Sigma$ the {\it fiber} of $T_{\phi}$.  
The $3$-manifold $T_{\phi}$ is a $\Sigma$-bundle over $S^1$ with the monodromy $\phi$.  
By Thurston \cite{Thurston98,Otal01}, 
$T_{\phi}$ admits a hyperbolic structure of finite volume 
if and only if $\phi$ is pseudo-Anosov. 

Thanks to the above result by Sakuma, 
one can define the non-empty subset 
$\mathcal{D}_g(M) \subset \mathrm{MCG}(\Sigma_g)$ consisting of elements 
$\phi \in  \mathrm{MCG}(\Sigma_g)$ such that 
$T_{\phi} \rightarrow M$ is a $2$-fold branched cover branched over a link, i.e., 
$$\mathcal{D}_g(M) = \{ \phi \in \mathrm{MCG}(\Sigma_g)\ |\ T_{\phi} \rightarrow M \ \mbox{is a $2$-fold branched cover}\}.$$
By the above result of Brooks, 
there exists a pseudo-Anosov element in $\mathcal{D}_g(M) $ if $g \ge \max(2, g(M))$, 
where $g(M)$ is the Heegaard genus of $M$.

To each pseudo-Anosov mapping class $\phi \in \mathrm{MCG}(\Sigma)$ on the surface $\Sigma= \Sigma_{g,p}$,  
there exists an associated dilatation (stretch factor) $\lambda(\phi)>1$ (\cite{FarbMargalit12}). 
The logarithm $\log(\lambda(\phi))$ of the dilatation  is called  the {\it entropy} of $\phi$. 
We call 
\begin{equation}
\label{equation_normalized-entropy}
\mathrm{Ent}(\phi)= |\chi(\Sigma)| \log(\lambda(\phi))
\end{equation}
the {\it normalized entropy} of $\phi$, 
where $\chi(\Sigma)$ is the Euler characteristic of $\Sigma$.

Consider the set 
$$\mathrm{Spec}(\Sigma) = \{ \log (\lambda(\phi))\ |\ \phi \in \mathrm{MCG}(\Sigma)\ \mbox{is\ pseudo-Anosov}\}.$$
For any subset of $\mathrm{Spec}(\Sigma)$, there exists a minimum. 
Then for any subset $G \subset \mathrm{MCG(\Sigma)}$ containing a pseudo-Anosov element, 
we set 
$$\ell(G)= \min\{\log (\lambda(\phi))\ | \   \phi \in G\ \mbox{is pseudo-Anosov}\},$$ 
that is the minimal entropy of pseudo-Anosov elements of $G$. 
Clearly we have $\ell(G) \ge \ell(\mathrm{MCG(\Sigma)})$. 
Penner \cite{Penner91} proved that 
$\ell(\mathrm{MCG}(\Sigma_g)) $ is comparable to $1/g$. 
Here we say that for two functions $A$ and $B$ with respect to $g$,  
$A$ is comparable to $B$ and write $A \asymp B$ 
if there exists a constant $C>0$ independent of $g$ so that 
$B/C \le A \le CB$. 

Asymptotic behaviors of minimal  entropies of various subgroups (subsets) of mapping class groups have been studied by many authors 
(\cite{FarbLeiningerMargalit08,Tsai09,Valdivia12,AgolLeiningerMargalit16,HiroseKin17,Yazdi18,HiroseIguchiKinKoda22}). 
For the hyperelliptic mapping class group $\mathcal{H}(\Sigma_g)$ defined on $\Sigma_g$,  
the minimal entropy $\ell(\mathcal{H}(\Sigma_g))$ for $\mathcal{H}(\Sigma_g)$  
is also comparable to $1/g$ (Hironaka-Kin \cite{HironakaKin06}). 
In contrast, the minimal entropy $\ell (\mathcal{I}(\Sigma_g))$ for the 
Torelli group $\mathcal{I}(\Sigma_g)$ defined on $\Sigma_g$ 
has a uniform lower bound (Farb-Leininger-Margalit \cite{FarbLeiningerMargalit08}).

Given a $3$-manifold $M$, 
we consider the subset $\mathcal{D}_g(M) \subset \mathrm{MCG}(\Sigma_g)$ and we write 
$$\ell_g(M)= \ell(\mathcal{D}_g(M)).$$  
Then  $\ell_g(M) \ge \ell(\mathrm{MCG}(\Sigma_g))$. 
The authors proved in \cite{HiroseKin202} that 
for the $3$-sphere $S^3$, it holds 
$\ell_g(S^3) \asymp \frac{1}{g}$. 
In this paper, we prove that there exist infinitely many closed $3$-manifolds with the same property as $S^3$. 
More precisely, we prove the following result. 

\begin{thm}
\label{thm_infinite}
There exist infinitely many closed orientable non-hyperbolic $3$-manifolds $M$ such that 
$\ell_g(M) \asymp \dfrac{1}{g}$. 
\end{thm}

For a link $L$ in $S^3$, 
let  $M_L \rightarrow S^3$ be  the $2$-fold branched cover of $S^3$ branched over a link $L$. 
Every link $L$ can be expressed  by the closure $\mathrm{cl}(b)$ of some braid $b$. 
Along the way in the proof of Theorem \ref{thm_infinite}, we prove in Theorem \ref{thm_application} that 
if $b$ is a homogeneous braid with certain conditions, then we have 
$\ell_g(M_{\mathrm{cl}(b)})  \asymp \dfrac{1}{g}$. 
The $3$-manifolds $M_{\mathrm{cl}(b)}$ with this property 
 include the following examples. 

\begin{itemize}
\item 
The lens space $L_{(2m,1)}$ of type $(2m,1)$ with $m \ne 0$ 
(Corollary \ref{corollary_lens}). 

\item 
The  connected sum $\sharp_{n}S^2 \times S^1$ of $n$ copies of $S^2 \times S^1$ for $n \ge 1$ 
(Theorem \ref{thm_trivial-link}). 

\item 
Dehn fillings of {\em minimally twisted $2k$-chain link\/} 
$\mathcal{C}_{2k}$, 
which is a $2k$ components chain link with every other link component lying flat 
in the plane of projection, and alternate link components to be perpendicular 
to the plane of projection. 
(See (3) of Figure \ref{fig_6-chain}  for $\mathcal{C}_6$.) 
\end{itemize}

Since $\mathcal{C}_{2k}$  is a hyperbolic link, all Dehn fillings (with a finite exceptions) are hyperbolic. 
Moreover $\mathrm{vol}(S^3 \setminus \mathcal{C}_{2k}) \ge 2k\, v_3$, 
where $v_3= 1.01494 \dots $ 
is the volume of the regular hyperbolic tetrahedron. 
Hence we have the following result. 

\begin{thm}
\label{thm_infinite-hyperbolic}
For any $R \geq 0$, there exists a closed orientable hyperbolic $3$-manifold $M$ 
with volume more than $R$ such that 
$\ell_g(M) \asymp \dfrac{1}{g}$. 
\end{thm}

Theorems \ref{thm_infinite} and \ref{thm_infinite-hyperbolic} imply that 
there exist infinitely many links $L$ in $S^3$ such that 
the minimal entropy $\ell_g(M_L) $  is comparable to $1/g$.  
Our conjecture is that every link in $S^3$ holds  this property:

\begin{conj}
\label{conj_sakuma-fiber}
For any link $L$ in $S^3$, we have 
$\ell_g(M_L) \asymp \dfrac{1}{g}$. 
\end{conj}

We ask the following question. 

\begin{ques}
Is there a closed orientable $3$-manifold $M$ 
such that the minimal entropy $\ell_g(M)$ has a uniform lower bound? 
\end{ques}

This paper is organized as follows. 
In Section \ref{section_backgrounds} 
we review basic facts on braids groups, mapping class groups and pseudo-Anosov mapping classes. 
In Section \ref{section_braids-increasing-middle} 
we introduce the notion of braids that are increasing in the middle. 
Then we combine some results in \cite{HiroseKin201,HiroseKin202} 
into new claims that can be used for the study of pseudo-Anosov elements in the set $\mathcal{D}_g(M_L)$ 
for each link $L$ in $S^3$.  
In Section \ref{section_applications} 
we prove Theorems \ref{thm_infinite} and \ref{thm_infinite-hyperbolic} 
and give some applications. 

\subsection*{Acknowledgments}  
We thank 
 Hirotaka Akiyoshi, Jessica Purcell and Han Yoshida for their information about the hyperbolic structures 
 on the complements of minimally twisted $2k$-chain links. 
 We thank Yuya Koda and Makoto Sakuma for helpful comments.

\section{Backgrounds and preliminaries} 
\label{section_backgrounds}

\subsection{Homogeneous braids and skew-palindromic braids} 
\label{subsection_homogeneous}

\begin{center}
\begin{figure}[t]
\includegraphics[height=4cm]{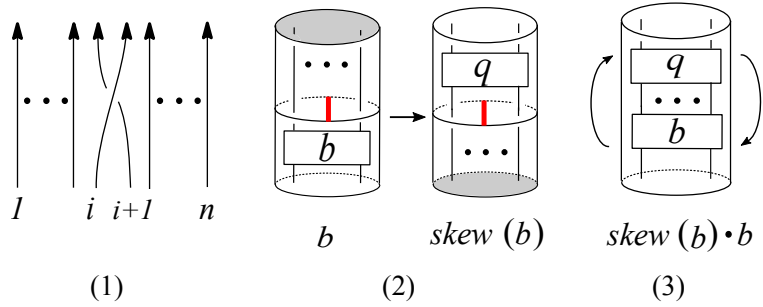}
\caption{(1) $\sigma_i \in B_n$. 
(2) The involution $R: D^2 \times [0,1] \to D^2 \times [0,1] $ 
with the fixed point set 
$\{( \pm r i, \frac{1}{2})\ |\ 0 \le r \le 1\} \subset D^2 \times \{\tfrac{1}{2}\}$. 
(3) The braid $\widetilde{b}= \mathrm{skew}(b) \cdot b$ that is invariant under the involution $R$. 
}
\label{fig_braids}
\end{figure}
\end{center}

\begin{center}
\begin{figure}[t]
\includegraphics[height=2.8cm]{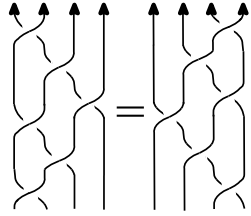}
\caption{A half twist $ \Delta_4 = \mathrm{skew}(\Delta_4) \in B_4$.} 
\label{fig_halftwist}
\end{figure}
\end{center}

Let $B_n$ be the (planar) braid group with $n$ strands. 
Let $a_1, \dots, a_n$ be the bottom end points of an $n$-braid $b \in B_n$. 
We call $a_i$'s the base points of $b$. 
We put indices $1, \dots, n$  to indicate the base points $a_1, \dots, a_n$ respectively.  
Let $\sigma_i$ ($i= 1, \dots, n$) denote the Artin generator of $B_n$ as  in Figure \ref{fig_braids}(1).

A braid word written by $\sigma_i^{\pm 1}$ ($i= 1, \dots, n-1$) 
is said to be {\it homogeneous} if 
for each $i \in \{1, \dots, n-1\}$, the exponents of all occurrences of $\sigma_i$ have the same sign. 
A braid $b$ is  said to be {\it homogeneous} if 
it can be represented by a homogeneous word. 
For example, 
the braid $\sigma_1 \sigma_3 \sigma_2^{-1} \sigma_3^2 \sigma_2^{-3}$ is homogeneous.

Now, we define an involution 
\begin{eqnarray*}
\mathrm{skew}: B_n &\to& B_n
\\
 \sigma_{n_1}^{\epsilon_1} \sigma_{n_2}^{\epsilon_2} \dots \sigma_{n_k}^{\epsilon_k} 
 &\mapsto& 
 \sigma_{n-n_k}^{\epsilon_k} \dots \sigma_{n- n_2}^{\epsilon_2} \sigma_{n-n_1}^{\epsilon_1}, 
\hspace{5mm} \epsilon_i = \pm 1.
\end{eqnarray*}
The map $\mathrm{skew}$ is an anti-homomorphism. 
A braid $b \in B_n$ is said to be {\it skew-palindromic} if $\mathrm{skew}(b)= b \in B_n$.

Note that 
$\mathrm{skew}: B_n \to B_n$ is induced by the  involution $R$ on the cylinder $D^2 \times [0,1]$: 
\begin{eqnarray*}
R: D^2 \times [0,1] &\to& D^2 \times [0,1], 
\\
(r e^{i \theta}, t) &\mapsto&  (r e^{i (\pi - \theta)}, 1-t)
\end{eqnarray*}
see Figure \ref{fig_braids}(2). 
Here we identify the disk $D^2$ with the unit disk centered at the origin in the complex plane ${\Bbb C}$. 

Notice that  the product $\mathrm{skew}(b) \cdot b \in B_n$  is a skew-palindromic braid  for any $b \in B_n$. 
We put 
$$\widetilde{b} := \mathrm{skew}(b) \cdot b,$$  
and we say that $\widetilde{b}$ is the {\it skew-palindromization} of $b$. 
See Figure \ref{fig_braids}(3).

\begin{ex}
\label{ex_homogeneous}
For a braid $b= \sigma_3^2 \sigma_4^{-2} \in B_5$, 
the skew-palindromization is 
$$\widetilde{b}= \mathrm{skew}(b) \cdot b = \sigma_1^{-2} \sigma_2^2 \sigma_3^2 \sigma_4^{-2} ,$$
that is a homogeneous braid. 
\end{ex}

Let $\Delta= \Delta_n \in  B_n$ be a half twist defined by 
\begin{eqnarray*}
\Delta
&=& (\sigma_1 \cdots \sigma_{n-1}) (\sigma_1 \cdots \sigma_{n-2}) \cdots (\sigma_1 \sigma_2) \sigma_1
\\
&=&
\sigma_{n-1}(\sigma_{n-2} \sigma_{n-1}) \cdots (\sigma_2 \cdots \sigma_{n-1}) (\sigma_1 \cdots \sigma_{n-1}). 
\end{eqnarray*}
See Figure \ref{fig_halftwist}. 
This means that 
$\Delta= \mathrm{skew}(\Delta)$, and hence 
$\Delta \in B_n$ is skew-palindromic for each $n$.  


\subsection{Dilatations and normalized entropies of braids}

Let $D_n$ be the $n$-punctured disk. 
We consider the mapping class group $\mathrm{MCG}(D_n)$, 
the group of isotopy classes of orientation preserving self-homeomorphisms on $D_n$ 
preserving  the boundary $\partial D$ of the disk setwise. 
There exists a surjective homomorphism 
$$\Gamma: B_n \rightarrow \mathrm{MCG}(D_n)$$
which sends each generator $\sigma_i$ to the right-handed half twist $h_i$ 
between the $i$-th and $(i+1)$-th punctures. 
Since the kernel of $\Gamma$ is isomorphic to the center $Z(B_n) = \langle \Delta^2 \rangle$ 
generated by a full twist $\Delta^2$, 
we have 
$$B_n/  \langle \Delta^2 \rangle   \simeq  \mathrm{MCG}(D_n).$$

Collapsing the boundary $\partial D$ to a puncture in the sphere $\Sigma_0$, 
we  have a homomorphism 
$$\mathfrak{c}: \mathrm{MCG}(D_n) \rightarrow \mathrm{MCG}(\Sigma_{0, n+1}).$$
We say that $b \in B_n$ is {\it periodic} (resp. {\it reducible}, {\it pseudo-Anosov}) if 
the mapping class $\mathfrak{c}(\Gamma(b)) $ is of the corresponding Nielsen-Thurston type. 

When $b \in B_n$ is pseudo-Anosov, 
we call $\lambda(b):= \lambda(\mathfrak{c}(\Gamma(b)) )$ the {\it dilatation} of $b$, and 
call $\mathrm{Ent}(b): = \mathrm{Ent}(\mathfrak{c}(\Gamma(b)) )$  the {\it normalized entropy} of $b$, 
see (\ref{equation_normalized-entropy}). 
By definition we have 
$$\mathrm{Ent}(b)= |\chi(\Sigma_{0,n+1})| \log(\lambda((\mathfrak{c}(\Gamma(b))))) = (n-1) \log(\lambda((\mathfrak{c}(\Gamma(b))))).$$

\subsection{Closures, braided links, circular plat closures of braids}
\label{subsection_closures}

\begin{center}
\begin{figure}[t]
\includegraphics[height=3.5cm]{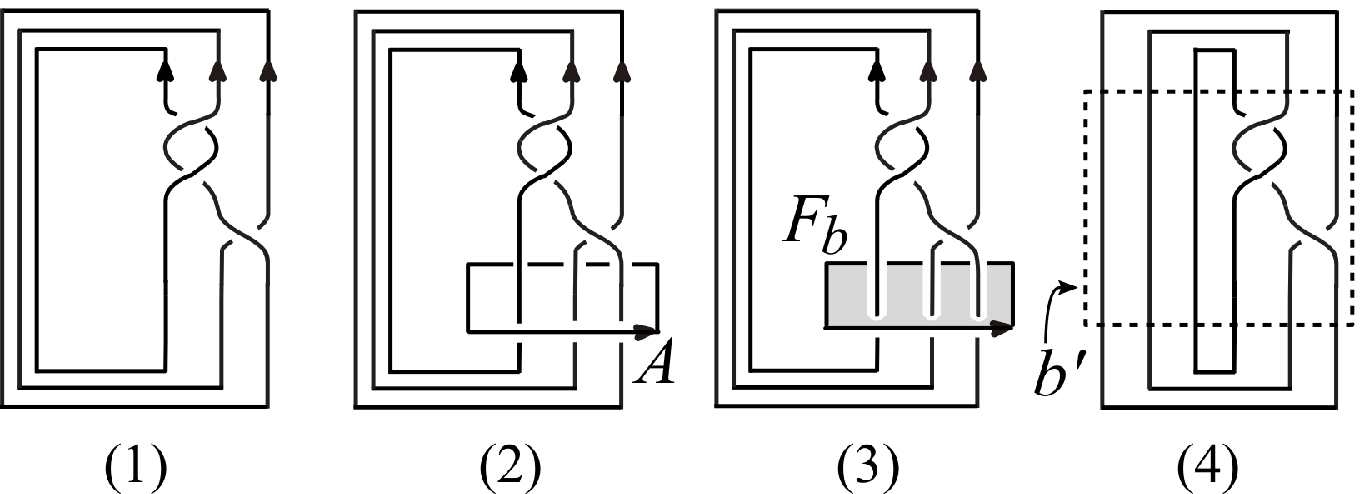}
\caption{Case $b= \sigma_1^2 \sigma_2^{-1} \in B_3$: 
(1) $\mathrm{cl}(b)$. (2) $\mathrm{br}(b)$. (3)  $F_b$. 
(4) $C(b')= \mathrm{cl}(b)$, where $b' = \sigma_4^2 \sigma_5^{-1} \in B_6$.}
\label{fig_closure}
\end{figure}
\end{center}

\begin{center}
\begin{figure}[t]
\includegraphics[height=4cm]{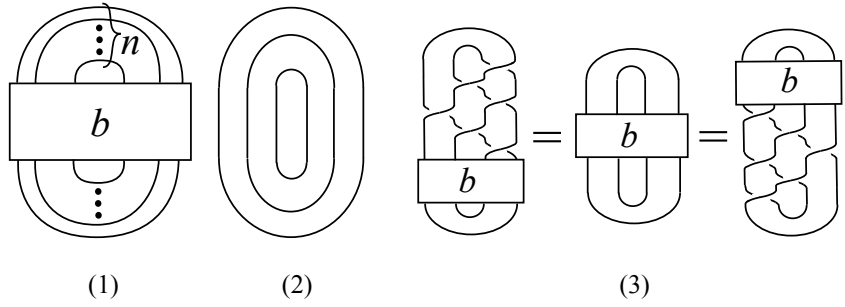}
\caption{ 
(1) $C(b)$ for $b \in B_{2n}$. 
(2) $E_3= C(e_6)$. 
(3) $C(\Delta_4 b) = C(b)= C(b \Delta_4)$ for $b \in B_4$.}
\label{fig:circular-plat}
\end{figure}
\end{center}

In this section we introduce three kinds of links in $S^3$, 
closures, braided links and circular plat closures obtained from planar braids.  
Given a link $L$ in a $3$-manifold $M$, 
we denote by $\mathcal{N}(L)$, a regular neighborhood of $L$. 
We denote by $\mathcal{E}(L)$, the exterior  $M \setminus \mathrm{int}(\mathcal{N}(L))$.

The {\it closure} $\mathrm{cl}(b)$ of $b$ is an oriented knot or link in the $3$-sphere $S^3$ 
whose orientation is induced by the strands of $b$, see Figure \ref{fig_closure}(1). 
The {\it braided link} 
$$\mathrm{br}(b) =A \cup \mathrm{cl}(b)$$ 
is a link in $S^3$ obtained from $\mathrm{cl}(b)$ with the braid axis $A$, see Figure \ref{fig_closure}(2). 
We think of $\mathrm{br}(b)$ as an oriented link in $S^3$ 
choosing an orientation of $A$ arbitrarily. 
(In Section \ref{subsection_i-increasing}, we assign an orientation of $A$ for {\it $i$-increasing braids}.) 
Let $T_b$ denote the exterior  of  the link $\mathrm{br}(b)$: 
$$T_b = \mathcal{E}(\mathrm{br}(b)) = S^3 \setminus \mathrm{int}(\mathcal{N}(\mathrm{br}(b))).$$
We define an $(n+1)$-holed sphere $F_b \subset T_b$ by 
$$F_b= D_A \setminus \mathrm{int}(\mathcal{N} (\mathrm{cl}(b))), $$
where $D_A$ is the disk bounded by the longitude of the regular neighborhood $\mathcal{N}(A)$ 
of the braid axis $A$ of $b$. 
See Figure \ref{fig_closure}(3). 
We give an orientation of $F_b$ which 
induces 
the orientation of $A$. 
The surface $F_b$ is a fiber of the fibration  $T_b \rightarrow S^1$  
and the braid $b$ determines the monodromy $\phi_b: F_b \to F_b$ (up to conjugation). 

The {\it circular plat closure} $ C(b)$ of  $b \in B_{2n}$ with even strands is an unoriented  knot or link in $S^3$ 
as in Figure \ref{fig:circular-plat}(1). 
For example,  the $n$-component trivial link $E_n$ is of the form 
$E_n= C(e_{2n})$, where $e_{2n} \in B_{2n}$ is the identity element, 
see Figure~\ref{fig:circular-plat}(2). 
It is not hard to see that 
the links $C(\Delta b)$, $C(b)$ and $C(b \Delta)$ are ambient isotopy to each other: 
\begin{equation}
\label{equation_halftwist}
C(\Delta b)= C(b)= C(b \Delta)
\end{equation}
as links in $S^3$. 
See Figure \ref{fig:circular-plat}(3).

\begin{remark}
\label{rem_circular-plat}
Any link $L$ in $S^3$ can be represented by the circular plat closure $C(b')$ 
for some braid $b'$ with even strands. 
To see this, we recall the fact that 
any link $L$ can be expressed by the closure $\mathrm{cl}(b)$ for some $b \in B_n$ ($n \ge 1$). 
The desired braid $b'$ with $2n$ strands can be obtained from the $n$-braid $b$ by adding $n$ straight strands: 
$b' = e_n \cup b \in B_{2n}$. 
Then we have $C(b') = \mathrm{cl}(b)= L$ 
as links in $S^3$.  
See Figure \ref{fig_closure}(4). 
\end{remark}

\subsection{A criterion to be pseudo-Anosov braids} 
\label{subsection_criterion}

In this section, we give a criterion for deciding planar braids to be pseudo-Anosov.

Given an oriented link $L= K_1 \cup \cdots \cup K_m$ with $m$ components in $S^3$,   
we denote by $\mathrm{lk}(K_i, K_j)$, 
the linking number between the two components $K_i$ and $K_j$. 
See \cite{Kawauchi96} for the definition of the linking number. 

Let 
$$\pi: B_n \rightarrow \mathfrak{S}_n$$ 
be a surjective homomorphism 
from the $n$-braid group $B_n$ to the permutation group $\mathfrak{S}_n$ of degree $n$ 
which sends  $\sigma_j$ to the transposition $(j, j+1)$. 
A braid $b \in B_ n$ is {\it pure} if $\pi(b)$ is the identity element of $\mathfrak{S}_n$ . 

For example, a $3$-braid $\beta= \sigma_1^4 \sigma_2^{-2}$ is pure. 
Let $\mathrm{cl}(\beta) = \ell_1 \cup \ell_2 \cup \ell_3$ be the closure of $\beta$, 
where  $\ell_i$ denotes the closure $\mathrm{cl}(\beta(i))$ of the $i$-th strand $\beta(i)$ 
with the base point $a_i$ for $i= 1,2,3$. 
Then 
$\mathrm{l k}(\ell_1, \ell_2)= 2$, $\mathrm{l k}(\ell_2, \ell_3)= -1$ and 
$\mathrm{l k}(\ell_3, \ell_1)= 0$.

\begin{prop}[Kobayashi-Umeda \cite{KobayashiUmeda10}]
\label{prop_KU}
Let $\beta  \in B_n$ be a pure braid for $n \ge 3$. 
Let $\mathrm{cl}(\beta) = \ell_1 \cup \dots \cup \ell_n$ be the closure of $\beta$, 
where $\ell_ i$ denotes the closure $\mathrm{cl}(\beta(i))$ of the $i$-th strand $\beta(i)$ with the base point $a_i$ 
for $i= 1, \dots, n$. 
\begin{enumerate}
\item[(1)]
Suppose that $\beta$ is  periodic. 
Then there exists an integer $n_0$ such that 
$\mathrm{l k}(\ell_i, \ell_j)= n_0$ for all $i,j$ with $i \ne j$. 

\item[(2)] 
Suppose that $\beta$ is reducible. 
Let $c$ be an inner most component of the system of the reducing curves for 
the mapping class $\Gamma(\beta) \in \mathrm{MCG}(D_n)$, 
and let $D_c$ be the disk bounded by $c$. 
Suppose that $a_s$ and $a_t$ are distinct base points in $D_c$. 
Then for each base point $a_j \not\in D_c$, 
the equality $\mathrm{l k}(\ell_j, \ell_s)=\mathrm{l k}(\ell_j, \ell_t)$ holds.  
\end{enumerate}
\end{prop}

The proof of the claim (1) (resp. the claim (2)) in Proposition \ref{prop_KU} can be found in 
\cite[Proposition 1]{KobayashiUmeda10} (resp. \cite[Proposition 2]{KobayashiUmeda10}). 
For the definition of {\it the system of the reducing curves} in the claim (2), 
see \cite{KobayashiUmeda10}, \cite[Chapter 13.2.2]{FarbMargalit12}

\begin{lem}
\label{lem_criterion}
Let $\beta  \in B_n$ be a pure braid for $n \ge 3$. 
Let $\ell_i$ ($i= 1, \dots, n$) be as in Proposition \ref{prop_KU}. 
Suppose that for any proper subset 
$$\mathcal{I}=\{i_1, \ldots, i_k\} \subsetneq \mathcal{J}:= \{1,2, \ldots,n\} $$  
consisting of $k$ distinct elements 
with $2 \le k < n$, 
there exist three elements $j \in \mathcal{J} \setminus  \mathcal{I}$ and $i_s, i_t \in \mathcal{I}$ 
such that 
$\mathrm{l k} (\ell_j, \ell_{i_s}) \ne \mathrm{l k} (\ell_j, \ell_{i_t}) $. 
Then $\beta$ is pseudo-Anosov. 
\end{lem}

\begin{proof}
By Proposition \ref{prop_KU}(1), 
the braid $\beta$ with the assumption of Lemma \ref{lem_criterion} can not be periodic. 
Assume that $\beta$ is reducible. 
Let $c$ be an inner most component of  the system of reducing curves for the mapping class $\Gamma(\beta) \in \mathrm{MCG}(D_n)$,   
and let $D_c$ be  the disk bounded by $c$. 
Let $a_{i_1}, \ldots, a_{i_k}$ be the set of all base points of $\beta$ contained in $D_c$. 
Then $\{1, \dots, n\} \setminus \{i_1, \dots, i_k\} \ne \emptyset$. 
By the assumption of Lemma \ref{lem_criterion}, 
there exist three elements 
$j \in \{1,2, \ldots,n\} \setminus  \{i_1, \ldots, i_k\}$ 
and $i_s, i_t \in \{i_1, \ldots, i_k\}$ such that 
$\mathrm{l k} (\ell_j, \ell_{i_s}) \ne \mathrm{l k} (\ell_j, \ell_{i_t}) $. 
By the choice of $j$, we have 
$a_j \not\in D_c$ for the base point $a_j$ of the strand $\beta(j)$ 
and $a_{i_s}, a_{i_t} \in D_c$. 
By Proposition \ref{prop_KU}(2), 
it must hold that 
$\mathrm{l k} (\ell_j, \ell_{i_s}) = \mathrm{l k} (\ell_j, \ell_{i_t}) $. 
This is a contradiction, and hence $\beta$ is not reducible. 
Since $\beta$ is neither periodic nor reducible, 
we conclude that $\beta$ is pseudo-Anosov.  
\end{proof}

\begin{lem} 
 \label{lem_torus-link}
Let  $b \in B_n$ be a pure braid for $n \ge 4$ of the form 
$$b= \sigma_{j_1}^{2m_1} \sigma_{j_2}^{2m_2} \dots \sigma_{j_k}^{2m_k},$$
where $m_1, \dots, m_k$ are non-zero integers and 
$j_1, \dots, j_k \in \{1, \dots, n-1\}$. 
Suppose that $b$ is homogeneous, 
and each $\sigma_i$ for $i= 1, \dots, n-1$ appears in $b$ at least once, 
i.e., 
$\{j_1, \dots, j_k\}= \{1, \dots, n-1\}$. 
Then $b$ is pseudo-Anosov. 
In particular, if 
$b=  \sigma_1^{2 m_1} \sigma_2^{2 m_2} \dots \sigma_{n-1}^{2 m_{n-1}} \in B_n$, 
then $b$ is pseudo-Anosov. 
\end{lem}

\begin{proof}
Let $\ell_i = \mathrm{cl}(b(i))$ ($i= 1, \dots, n$) be the component of $\mathrm{cl}(b)$ as in Proposition \ref{prop_KU}. 
The assumption of Lemma \ref{lem_torus-link} means that 
$\mathrm{l k} (\ell_i, \ell_j) \ne 0$ if and only if $|i-j|=1$. 
It is sufficient to prove the following: 
For any proper subset $\mathcal{I}= \{i_1, \ldots, i_k\} \subsetneq \mathcal{J}= \{1,2, \ldots,n\} $ with $2 \le k < n$, 
there exist three elements 
$j \in \mathcal{J} \setminus \mathcal{I}$ and $i_s, i_t \in \mathcal{I} $  such that 
\begin{equation}
\label{equation_propersubset}
|j-i_s|= 1 \hspace{2mm}  \mbox{and} \hspace{2mm}  |j-i_t|> 1, 
\end{equation}
i.e., 
$\mathrm{l k} (\ell_j, \ell_{i_s}) \ne 0$ and $\mathrm{l k} (\ell_j, \ell_{i_t}) =0$.  
Then $\mathrm{l k} (\ell_j, \ell_{i_s}) \ne \mathrm{l k} (\ell_j, \ell_{i_t}) $, 
and Lemma \ref{lem_criterion} tells us that $b$ is pseudo-Anosov.

Since $\mathcal{I}$ is a proper subset of $\mathcal{J}$, there are 
$i_u \in \mathcal{I}$ and $h \in \mathcal{J} \setminus \mathcal{I}$ such that $|i_u -h | =1$. 
Moreover we can take an element $i_v \in \mathcal{I}$ such that $i_v \not= i_u$. 
It is possible to take such $i_v \in \mathcal{I}$ 
because $|\mathcal{I}| \geq 2$, 
where $|S|$ denotes the cardinality of the finite set $S$. 
In case where $|i_v -h|>1$,  
the three elements $j:= h$, $i_s:= i_u$ and $i_t:= i_v$ satisfy (\ref{equation_propersubset}). 
In case where $|i_v -h|=1$,  
the three elements $i_u, h, i_v$ are consecutive integers. 
Without loss of generality, we may assume that $i_u < h < i_v$. 
Since $n \geq 4$,  the following cases occur:  
(1) $1= i_u < h < i_v < n$, (2) $1< i_u < h < i_v < n$, (3) $1< i_u < h < i_v = n$. 
In cases (1) and (2), we have $i_v + 1 \in \mathcal{J}$. 
If $i_v + 1 \in \mathcal{I}$, 
then $j:= h$, $i_s:= i_u$ and $i_t:= i_v+1$ satisfy (\ref{equation_propersubset}). 
If $i_v + 1 \not\in \mathcal{I}$, 
then 
$j:= i_v+1$, $i_s:= i_v$ and $i_t:= i_u$ satisfy (\ref{equation_propersubset}). 
In case (3), we can choose three elements $j$, $i_s$ and $i_t$ 
that satisfy (\ref{equation_propersubset}) in the same way as above. 
This completes the proof. 
\end{proof}

\begin{ex}
By Lemma \ref{lem_torus-link}, 
the braid 
$\widetilde{b}=  \sigma_1^{-2} \sigma_2^2   \sigma_3^2 \sigma_4^{-2} \in B_5$ 
given in  Example \ref{ex_homogeneous}  is pseudo-Anosov 
\end{ex}

\subsection{Branched virtual fibering theorem}

\begin{center}
\begin{figure}[t]
\includegraphics[height=3.5cm]{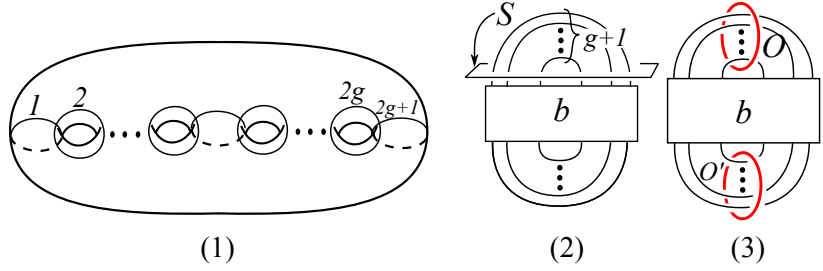}
\caption{(1) Simple closed curves labeled $1, \dots, 2g+1$ in $\Sigma_g$. 
(2) A $(g+1)$-bridge sphere $S$ of $C(b)$ and 
(3) the link $C(b) \cup W$ for $b \in B_{2g+2}$, where $W= O \cup O'$.}
\label{fig:hyperelliptic}
\end{figure}
\end{center}

We recall the branched virtual fibering theorem due to Sakuma \cite{Sakuma81}. 
See also Koda-Sakuma  \cite[Theorem 9.1]{KS22}.

\begin{thm}
\label{thm_Sakuma}
Let $M$ be a closed orientable $3$-manifold. 
Suppose that $M$ admits a genus $g$ Heegaard splitting. 
Then there exists a $2$-fold branched cover $\widetilde{M}$ of $M$ 
which is a $\Sigma_g$-bundle over the circle. 
\end{thm}

In \cite{HiroseKin202}, the authors gave an alternative construction of surface bundles over the circle 
in Sakuma's result when closed $3$-manifolds are $2$-fold branched covers of $S^3$  branched over links. 
We recall our construction in this section. 

Let $\tau_i$ denote the right-handed Dehn twist 
about the simple closed curve labeled $i$ in Figure \ref{fig:hyperelliptic}. 
There exists a homomorphism 
from the braid group $B_{2g+2}$ to the mapping class group $\mathrm{MCG}(\Sigma_g)$ 
$$\mathfrak{t}: B_{2g+2} \rightarrow \mathrm{MCG}(\Sigma_g)$$ 
which sends $\sigma_i$ to $\tau_i$ for $i=1, \ldots, 2g+1$, 
since $\mathrm{MCG}(\Sigma_g)$ has the same braid relation as $B_{2g+2}$. 
Notice that its image $\mathfrak{t}(B_{2g+2})$ is the {\it hyperelliptic mapping class group} $\mathcal{H}(\Sigma_g)$. 
This is the subgroup of $\mathrm{MCG}(\Sigma_g)$  
consisting of elements with representative homeomorphisms 
that commute with some fixed hyperelliptic involution on $\Sigma_g$.

Let $L$ be a link in $S^3$. 
By Remark \ref{rem_circular-plat} 
we may suppose that $L$ is of the form $L= C(b)$ for some $b \in B_{2g+2}$. 
Let 
$$q= q_{L}: M_L \rightarrow S^3$$ 
denote the $2$-fold  branched covering map of $S^3$ 
branched over $L$. 
We have a $(g+1)$-bridge sphere $S$ for $L=C(b)$ as in Figure \ref{fig:hyperelliptic}(2).  
The $3$-manifold $M_L$ admits a  genus $g$ Heegaard splitting 
with the Heegaard surface $q^{-1}(S)$. 
Consider the trivial link $W= O \cup O'$ with $2$ components and the link 
$C(b) \cup W $ in $S^3$ 
as shown in Figure~\ref{fig:hyperelliptic}(3). 
Then we have the following result.

\begin{thm}[Theorem B in \cite{HiroseKin202}]
\label{thm_2-fold}
Let $q: M_L \rightarrow S^3$ be the $2$-fold branched covering map of $S^3$ branched over a link $L= C(b)$ 
for a braid  $b \in B_{2g+2}$. 
Consider the skew-palindromization $\widetilde{b} $ of $b$ 
and the mapping class $\mathfrak{t}(\widetilde{b}) \in \mathcal{H}(\Sigma_g) \subset \mathrm{MCG}(\Sigma_g)$. 
Then $T_{\mathfrak{t}(\widetilde{b})} \rightarrow M_L$ is a $2$-fold branched cover of $M_L$ 
branched over the link $q^{-1}(W)$. 
In particular 
$\mathfrak{t}(\widetilde{b}) \in \mathcal{D}_g(M_L)$. 
\end{thm}

\begin{proof}[Sketch of Proof] 
We regard $\widetilde{M_L}$ as the 	
$\mathbb{Z}/2\mathbb{Z} +  \mathbb{Z}/2\mathbb{Z}$-cover 
of $S^3$ branched over the link $C(b) \cup W$ associated with the epimorphism 
$$H_1(S^3 \setminus (C(b) \cup W) )\to  \mathbb{Z}/2\mathbb{Z} +  \mathbb{Z}/2\mathbb{Z}$$  
which maps the meridians of $C(b)$ to $(1,0)$ and the meridians of $W$ to $(0,1)$. 
Let $q_W : M_W \to S^3$ be the $2$-fold  branched covering map of $S^3$ branched over the link $W$. 
Note that $M_W = S^2 \times S^1$. 
Then $\widetilde{M_L}$ is the $2$-fold  branched cover of $S^2 \times S^1$ 
branched over the link $q_W^{-1} (C(b)) = \mathrm{cl}(\widetilde{b})$ associated with the epimorphism 
$$H_1(S^2 \times S^1 \setminus \mathrm{cl}(\widetilde{b})) \to  \mathbb{Z}/2\mathbb{Z}$$ 
which maps the meridians of  $\mathrm{cl}(\widetilde{b})$ to $1$ and $\{ \mathrm{pt} \} \times S^1$ to $0$. 
Therefore, $\widetilde{M_L}$ is homeomorphic to the mapping torus 
$T_{\mathfrak{t}(\widetilde{b})}$ of $\mathfrak{t}(\widetilde{b})$.  
This completes the proof.
\end{proof}

We are interested in the case where the mapping class $\mathfrak{t}(\widetilde{b}) \in  \mathrm{MCG}(\Sigma_g)$ 
given in Theorem \ref{thm_2-fold} 
is pseudo-Anosov. 
The following lemma will be used in the later section. 

\begin{lem}[Lemma 5 in \cite{HiroseKin202}]
\label{lem_disk}
Let $\beta \in B_{2g+2}$ be a pseudo-Anosov braid and let 
$\Phi_{\beta} : D_{2g+2} \rightarrow D_{2g+2}$ be a pseudo-Anosov homeomorphism 
which represents $\Gamma(\beta) \in \mathrm{MCG}(D_{2g+2})$. 
Suppose that the pseudo-Anosov braid $\beta$ possesses the following condition:  
\begin{quote}
$\diamondsuit$  
The stable foliation $\mathcal{F}$ for $\Phi_{\beta}$ defined on $D_{2g+2}$ 
is not $1$-pronged at the boundary $\partial D$ of the disk. 
\end{quote}
Then $\mathfrak{t}(\beta) \in \mathrm{MCG}(\Sigma_g)$ is pseudo-Anosov, 
and the equality $\lambda(\mathfrak{t}(\beta) ) = \lambda(\beta)$ holds. 
\end{lem}

The basic facts on (un)stable foliations for pseudo-Anosov homeomorphisms 
can be found in Chapter 11.2 and Chapter 13 in \cite{FarbMargalit12}.

 \subsection{Thurston norm}
 \label{subsection_thurstonnorm}

 Let $M$ be a $3$-manifold with boundary  (possibly $\partial M = \emptyset$). 
When $M$ is a hyperbolic $3$-manifold, 
there exists a norm $\| \cdot \|$ on $H_2(M, \partial M; {\Bbb R})$, that is called the Thurston norm~\cite{Thurston86}. 
The norm $\| \cdot \|$ has the property such that 
for any integral class $a  \in H_2(M, \partial M; {\Bbb R})$, we have 
$$\|a\|= \min_S \{- \chi(S)\},$$ 
where 
 the minimum is taken over all  oriented surface $S$ embedded in $M$ with $a= [S]$ 
and with no components  of non-negative Euler characteristic. 
The following result by Thurston describes a relation between the norm $\|\cdot\|$ and fibrations on $M$.

\begin{thm}[Thurston~\cite{Thurston86}]
\label{thm_norm}
The norm $\| \cdot \|$ on $H_2(M, \partial M; {\Bbb R})$ has the following properties. 
\begin{enumerate}
\item[(1)]
There exist a set of maximal open cones 
$\mathscr{C}_1, \dots, \mathscr{C}_k$ in $H_2(M, \partial M; {\Bbb R})$ 
and a bijection between the set of isotopy classes of connected fibers of fibrations $M \rightarrow S^1$ 
and the set of primitive integral classes in $\mathscr{C}_1 \cup \cdots \cup \mathscr{C}_k$. 

\item[(2)] 
The restriction of $\|\cdot\|$ to $\mathscr{C}_j$ is linear for each $j = 1, \dots, k$. 

\item[(3)] 
For a fiber $F_a$  of the fibration $M \rightarrow S^1$ associated with a primitive integral class $a$ in $\mathscr{C}_j$ for $j= 1, \dots, k$, 
we have $\|a\| = - \chi (F_a)$. 
\end{enumerate}
\end{thm}
We call the open cones $\mathscr{C}_j$  {\it fibered cones} of $M$.

\begin{thm}[Fried~\cite{Fried82}] 
\label{thm_Fried_1}
For a fibered cone $\mathscr{C}$ of a hyperbolic $3$-manifold $M$, 
there exists a continuous function 
$\mathrm{ent}: \mathscr{C}  \rightarrow {\Bbb R}$ 
with the following properties. 
\begin{enumerate}

\item[(1)] 
For the monodromy  $\phi_a: F_a \rightarrow F_a$ of the fibration $M \rightarrow S^1$ 
associated with a primitive integral class $a \in \mathscr{C}$, 
we have 
$\mathrm{ent}(a) = \log (\lambda(\phi_a))$, 
i.e., 
$\mathrm{ent}(a)$ equals the entropy of the pseudo-Anosov monodromy $\phi_a$.

\item[(2)] 
$\mathrm{Ent}= \| \cdot\| \mathrm{ent}: \mathscr{C} \rightarrow  {\Bbb R}$  
is a continuous function which becomes constant on each ray through the origin.

\end{enumerate}
\end{thm}

\noindent
We call $\mathrm{ent}(a)$ and $\mathrm{Ent}(a)$ the {\it entropy} and  {\it normalized entropy} of 
the class $a \in \mathscr{C}$.  
By Theorem \ref{thm_norm}(3) and Theorem \ref{thm_Fried_1}(1), if $a \in \mathscr{C}$ is a primitive integral class, then 
$$\mathrm{Ent}(a)= \|a\| \mathrm{ent}(a)= |\chi(F_a)| \log (\lambda(\phi_a)) (= \mathrm{Ent}(\phi_a)).$$

\subsection{$i$-increasing braids}
\label{subsection_i-increasing}

\begin{center}
\begin{figure}
\includegraphics[width=4in]{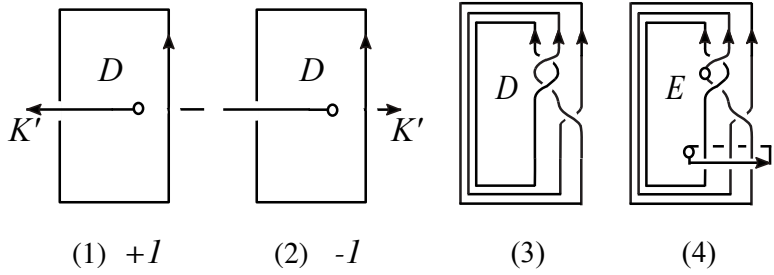} 
\caption{The sign of the intersection point:  
(1) $+1$ and (2) $-1$. 
(3) The associated disk $D = D_{(b,1)}$ and 
(4) the  surface $E= E_{(b,1)}$ 
for a $1$-increasing braid $b= \sigma_1^2 \sigma_2^{-1} \in B_3$. 
($E_{(b,1)}$ is a twice punctured disk in this case.)}
\label{fig_intersection}
\end{figure}
\end{center}

In \cite{HiroseKin201}, the authors introduced {\it $i$-increasing braids}. 
In this section, we review some properties of  $i$-increasing braids 
that are needed in the later section. 

Recall that  $\pi: B_n \rightarrow \mathfrak{S}_n$ is the surjective homomorphism 
as in Section \ref{subsection_criterion}. 
We denote by $\pi_b$, the permutation $\pi(b) \in \mathfrak{S}_n$ for $b \in B_n$. 
Suppose that $b \in B_n $ is a braid with $\pi_b(i)= i$, 
i.e., the permutation $\pi_b$ fixes the index $i$. 
The closure $\mathrm{cl}(b(i))$ of the $i$-th strand $b(i)$ 
 is a component of the closure $\mathrm{cl}(b)$ of $b$.  
We consider an oriented disk $D= D_ {(b,i)}$ bounded by the longitude $\ell_i$ of 
a regular neighborhood $\mathcal{N}(\mathrm{cl}(b(i)))$ of $\mathrm{cl}(b(i))$. 
Such a disk $D$ is unique up to isotopy on $\mathcal{E}(\mathrm{cl}(b(i)))$. 
Let $b-b(i) \in B_{n-1}$ be a braid with $n-1$ strands obtained from $b$ by removing the $i$-th strand $b(i)$. 
The braid $b$ is said to be  {\it $i$-increasing} (resp. {\it $i$-decreasing}) 
if there exists a disk $D= D_{(b,i)} $ as above 
with the following conditions (D1) and (D2). 
\begin{enumerate}
\item[(D1)]
There exists at least one component $K'$ of $\mathrm{cl}(b-b(i))$ 
such that $D \cap K' \ne \emptyset$. 

\item[(D2)]
Each component of $\mathrm{cl}(b-b(i))$ and $D$ intersect with each other transversally, 
and every intersection point has the same sign $+1$ (resp. $-1$), see Figure~\ref{fig_intersection}(1)(2).  
\end{enumerate}
We call $D= D_ {(b,i)}$ the {\it associated disk} of the pair $(b,i)$. 
Then we set 
$$I(b,i)= D \cap \mathrm{cl}(b-b(i)).$$ 
By (D1) we have $I(b,i) \ne \emptyset$. 
Let $u(b,i) \ge 1$ be the cardinality $|I(b,i)|$ of  $I(b,i)$. 
We call $u(b,i)$ the {\it intersection number} of the pair $(b,i)$.

\begin{ex}
\label{ex_increasing}\ 
\begin{enumerate}
\item[(1)] 
A braid $b= \sigma_1^2 \sigma_2^{-1} \in B_3$ is  $1$-increasing with $u(b, 1)=1$. 
See Figure \ref{fig_intersection}(3). 

\item[(2)] 
 A pure braid  $b= \sigma_1^4 \sigma_2^{-2} \in B_3$  
is $1$-increasing with $u(b, 1)=2$ and $3$-decreasing with $u(b,3)=1$. 
\end{enumerate}
\end{ex}

Properties of $i$-increasing braids are given in the next two lemmas. 
The same properties hold for  $i$-decreasing braids.

\begin{lem}
\label{lem_product}
If $b$ and $b'$ are $i$-increasing braids with the same number of strands, 
then the product $b b'$ is also $i$-increasing 
such that $u(b b', i)= u(b,i)+ u(b',i)$. 
\end{lem}

\begin{proof}
Let $D^2$ be the disk with radius $1$. 
For $k= 0,1,2$, 
let $(D^{2} \times [0,1])_{k}$  be cylinders in $S^{3}$ such that 
$(D^{2} \times [0,1])_{1} \cap \mathrm{cl}(b) = b$, $(D^{2} \times [0,1])_{0} \cap \mathrm{cl}(b') = b'$ 
and $(D^{2} \times [0,1])_{2} \cap \mathrm{cl}(bb') = bb'$. 
We set $R_{\theta}:= \{r e^{i\theta}\ |\ 0 \le r \le 1\} \subset D^{2}$. 
We denote by $D$ (resp. $D'$), an associated disk of the pair $(b,i)$ (resp. $(b',i)$). 
By an ambient isotopy, we may assume that 
there exists $\theta_0 \in [0, 2 \pi)$ such that 
$D \cap (D^{2} \times [0,1])_{1} = R_{\theta_0} \times [0,1]$ 
(resp. $D' \cap (D^{2} \times [0,1])_{0} = R_{\theta_0} \times [0,1]$) and 
$D \cap \mathrm{cl}(b) = R_{\theta_0} \times [0,1] \cap \mathrm{cl}(b)$ 
(resp. $D' \cap \mathrm{cl}(b') = R_{\theta_0} \times [0,1] \cap \mathrm{cl}(b')$). 
We stuck the cylinder $(D^{2} \times [0,1])_{1}$ over the cylinder $(D^{2} \times [0,1])_{0}$ so that 
$(D^{2} \times \{ 0 \})_{1}$ is attached to $(D^{2} \times \{ 1 \})_{0}$, and 
we identify the result with $D^{2} \times [0,2]$. 
Let $$F : D^{2} \times [0,2] \to (D^{2} \times [0,1])_{2}$$ 
be the homeomorphism defined by $F(x,t) = (x,t/2)$. 
Then $F( b \cup b') = b b'$ by the definition of 
the product of braids. 
The image of the union of $R_{\theta_0} \times [0,1]$ in $(D^{2} \times [0,1])_{1}$ 
and $R_{\theta_0} \times [0,1]$ in $(D^{2} \times [0,1])_{0}$ 
under the homeomorphism $F$ is 
$R_{\theta_0} \times [0,1] \subset (D^{2} \times [0,1])_{2}$, which intersects 
with the closure $\mathrm{cl}(b b')$ of the product $bb'$ positively at  $(u(b,i) + u(b',i))$ points. 
Therefore, $b b'$ is $i$-increasing and the equality $u(b b', i) = u(b,i) + u(b',i)$ holds. 
\end{proof}


\begin{lem}
\label{lem_involution-skew}
Suppose that $b \in B_n$ is an $i$-increasing braid.  
Then $\mathrm{skew}(b) \in B_n$ is an $(n-i+1)$-increasing braid 
such that $u(\mathrm{skew}(b),n-i+1)= u(b,i)$. 
\end{lem}

\begin{proof}
The assertion follows from that fact that 
$\mathrm{skew}: B_n \to B_n$ is  induced by the  involution $R$ on the cylinder $D^2 \times [0,1]$ 
given in Section \ref{subsection_homogeneous}. 
\end{proof}

\begin{ex}
\label{ex_3increasing}
Let $b= \sigma_3^2 \sigma_4^{-2} $ be the $5$-braid as in Example \ref{ex_homogeneous}. 
Then $b$ is $3$-increasing with $u(b,3)=1$. 
By Lemma \ref{lem_involution-skew}, the braid 
$\mathrm{skew}(b) =\sigma_1^{-2} \sigma_2^2 $ is 
$3$-increasing with $u(\mathrm{skew}(b), 3)= 1$. 
Then by Lemma \ref{lem_product}, 
the braid $\widetilde{b}= \mathrm{skew}(b) \cdot b =\sigma_1^{-2} \sigma_2^2  \sigma_3^2 \sigma_4^{-2} $ is also $3$-increasing 
with $u(\widetilde{b}, 3)= u(\mathrm{skew}(b), 3)+ u(b,3)= 2$. 
\end{ex}

Recall that $T_b= \mathcal{E}(\mathrm{br}(b))$ is the exterior of the braided link $\mathrm{br}(b)$ 
and the surface $F_b$ is a genus $0$ fiber of the fibration $T_b \rightarrow S^1$. 
See Section \ref{subsection_closures}. 
We shall define the $2$-dimensional subcone 
$C_{(b,i)}$ of $H_2(T_b, \partial T_b; {\Bbb R})$ for an $i$-increasing braid $b$. 
To do this, 
we first consider the braided link $\mathrm{br}(b)= \mathrm{cl}(b) \cup A$. 
The associated disk $D = D_{(b,i)}$ has a unique point of intersection with  $A$, 
and the cardinality of $I(b,i) \cup (D \cap A) $ is  $u(b,i)+1$. 
To deal with $\mathrm{br}(b)= \mathrm{cl}(b) \cup A$ as an oriented link, 
we consider an orientation of $\mathrm{cl}(b)$ as we described in Section \ref{subsection_closures}, and 
assign an orientation of the braid axis $A $ of $b$  
so that the sign of the intersection between  $D$ and $A$ is $+1$ as in Figure \ref{fig_intersection}(1).  
See Figure~\ref{fig_closure}(2)  
for the orientation of $A$ of the $3$-braid $\sigma_1^2 \sigma_2^{-1}$ that is $1$-increasing.

Next, we define an oriented surface $E_{(b,i)}$ of genus $0$ embedded in $T_b$. 
Consider small $u(b,i)+1$ disks in the associated disk $D=D_{(b,i)}$ whose centers are points of $I(b,i) \cup (D \cap A) $. 
Then  $E_{(b,i)} $ is a surface of genus $0$ with $u(b,i)+2$ boundary components 
obtained from $D$ by removing the interiors of those small disks. 
We choose the orientation of $E_{(b,i)}$ so that it agrees with the orientation of $D$. 
See Figure \ref{fig_intersection}(4).

Lastly, we define the $2$-dimensional subcone 
$C_{(b,i)}$ of $H_2(T_b, \partial T_b; {\Bbb R})$ spanned by 
the two integral classes $[F_b]$ and $[E_{(b,i)}]$ as follows. 
\begin{equation}
\label{equation_subcone}
C_{(b,i)} = \{x[F_b]+ y[E_{(b,i)}]\ |\ x>0, \ y > 0\} .
\end{equation}
We write $(x,y)= x[F_b]+ y[E_{(b,i)}] \in C_{(b,i)}$. 
The Thurston norm of $(x,y)$ is denoted by $\| (x,y)\|$. 

\begin{thm}[\cite{HiroseKin201}]
\label{thm_disktwist} 
Let $b$ be an $i$-increasing braid. 
Suppose that $b$ is  pseudo-Anosov. 
Let $\mathscr{C}$ be the fibered cone of the  $3$-manifold $T_b$ containing 
$[F_b] = (1,0) \in C_{(b,i)}$.  
Then we have the following. 
\begin{enumerate}
\item[(1)] 
$C_{(b,i)} \subset \mathscr{C}$. 

\item[(2)] 
The fiber $F_{(x,y)}$ for each primitive integral class $(x,y) \in C_{(b,i)}$ has genus $0$. 

\item[(3)] 
Let $\phi_{(x,y)}: F_{(x,y)} \rightarrow F_{(x,y)}$  
denote the monodromy of the fibration $T_b \rightarrow S^1$ 
associated with a primitive integral class $(x,y) \in C_{(b,i)} $. 
Then there exists a $j$-increasing braid $b_{(x,y)} \in B_{ \|(x,y)\|+1}$  for some index $j= j_{(x,y)}$ 
which gives the monodromy $\phi_{(x,y)}: F_{(x,y)} \rightarrow F_{(x,y)}$. 
\end{enumerate}
\end{thm}

The proof of Theorem \ref{thm_disktwist}(1)(2) 
can be found in  Theorem 3.2(1)(2)  in \cite{HiroseKin201}. 
The statement of Theorem \ref{thm_disktwist}(3) 
follows from the argument in the proof of Theorem 3.2(3) in \cite{HiroseKin201}.

\section{Braids increasing in the middle}
\label{section_braids-increasing-middle}

Let $b$ be a braid with $2n+1$ strands. 
Then the notion `$i$-increasing braid' makes sense for $i= 1, \dots, 2n+1$. 
(See Section \ref{subsection_i-increasing}.) 
In this section, we restrict our attention to the case $i= n+1$: 
Suppose that 
$b$ is an $(n+1)$-increasing braid.  
In this case we say that $b$ is  {\it increasing in the middle}.  
We write 
$C_b:= C_{(b, n+1)}$ for the subcone of $H_2(T_b, \partial T_b; {\Bbb R})$. 
(See (\ref{equation_subcone}) for the definition of the subcone.)  
Then $b^{\bullet} \in B_{2n}$ denotes the braid 
 obtained from $b \in B_{2n+1}$ 
by removing the strand of the middle index $n+1$.

\begin{ex}(cf. Example \ref{ex_3increasing}) 
\label{ex_product}
Suppose that $b \in B_ {2n+1}$ is a braid increasing in the middle. 
By Lemma \ref{lem_involution-skew}, the braid $\mathrm{skew}(b)$ is increasing in the middle. 
By Lemma \ref{lem_product} the braid $ \widetilde{b}= \mathrm{skew}(b) \cdot b$ 
is also increasing in the middle with the intersection number 
$$u(\widetilde{b},n+1)= u(\mathrm{skew}(b), n+1)+ u(b,n+1)= 2u(b,n+1). $$
The skew-palindromization $\widetilde{(b^{\bullet})} $ of $b^{\bullet}$ 
satisfies 
$$\widetilde{(b^{\bullet})} = \mathrm{skew}(b^{\bullet}) b^{\bullet}=  (\widetilde{b})^{\bullet},$$ 
i.e., $\widetilde{(b^{\bullet})} \in B_{2n}$ is obtained from the skew-palindromization $\widetilde{b} $ of $b $  
by removing the strand of the middle index $n+1$.    
Hereafter we simply denote the braid $\widetilde{(b^{\bullet})} $ 
by $\widetilde{b^{\bullet}}$. 
Applying Theorem \ref{thm_2-fold} to the circular plat closure $L= C(b^{\bullet})$ of $b^{\bullet} \in B_{2n}$, 
we have $\mathfrak{t}(\widetilde{b^{\bullet}}) \in \mathcal{D}_{n-1}(M_{C(b^{\bullet})})$. 
\end{ex}

\begin{ex}
\label{ex_halftwist}
The half twist $\Delta = \Delta_{2n+1} \in B_{2n+1}$ is a braid increasing in the middle 
with  $u(\Delta,n+1)=n$. 
By Lemma \ref{lem_product} 
 the positive power $\Delta^p$ for $p \ge 1$ is a braid increasing in the middle, and 
it holds 
$$u(\Delta^p,n+1)= p \cdot u(\Delta,n+1)=  pn.$$
The braid $\Delta^{\bullet} \in B_{2n}$ satisfies 
$\Delta^{\bullet}= \Delta_{2n} \in B_{2n}$, 
i.e.,  $\Delta^{\bullet}$ is equal to the half twist $\Delta_ {2n}$ with $2n$ strands. 
 \end{ex}

\begin{center}
\begin{figure}
\includegraphics[height=5cm]{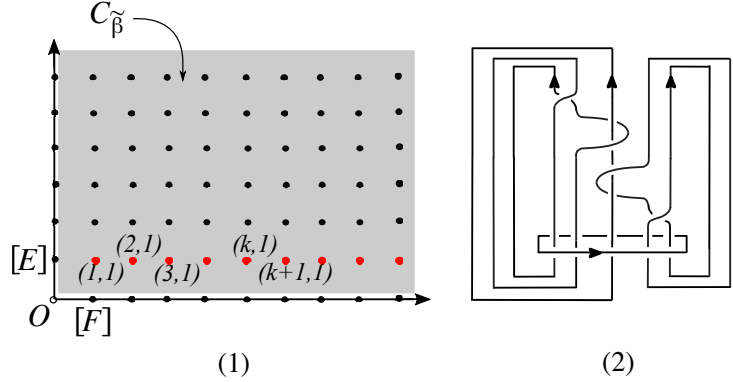} 
\caption{
 (1) The subcone $ C_{\widetilde{\beta}}= C_ {(\widetilde{\beta}, n+1)}$ spanned by $[F]$ and $[E]$, 
 where $F:= F_{\widetilde{\beta}}$ and $E:= E_{(\widetilde{\beta},n+1)}$. 
 (Primitive integral classes $(1,1), (2,1), \dots, (k,1), \dots$ are indicated.) 
 (2) The braided link $\mathrm{br}(\widetilde{\beta})$ of 
 $\widetilde{\beta}= \sigma_1 \sigma_2^2 \sigma_3^2 \sigma_4 \in B_5$.} 
\label{fig_subcone_ver2}
\end{figure}
\end{center}

Given a braid $\beta \in B_{2n+1}$ increasing in the middle, 
we suppose that the skew-palindromization $\widetilde{\beta}$ of $\beta$ is pseudo-Anosov. 
Consider the subcone $C_{\widetilde{\beta}} = C_{(\widetilde{\beta}, n+1)} $ of 
$H_2(T_{\widetilde{\beta}}, \partial T_{\widetilde{\beta}}; {\Bbb R})$ 
for the hyperbolic $3$-manifold $T_{\widetilde{\beta}}$ (Figure \ref{fig_subcone_ver2}(1)). 
We now apply Theorem \ref{thm_disktwist} for the  skew-palindromization $\widetilde{\beta}$. 
For the class $(1,0)= [F_{\widetilde{\beta}}] \in C_{\widetilde{\beta}} $, 
the monodromy $\phi_{(1,0)}: F_{(1,0)} \rightarrow F_{(1,0)}$ 
defined on the fiber $F_{(1,0)} = F_{\widetilde{\beta}}$ 
is given by the braid $\widetilde{\beta}$ 
that is increasing in the middle. 
The following result (Theorem \ref{thm_fibered-cone}) tells us that 
this property is inherited for all primitive classes $(x, y) \in C_{\widetilde{\beta}} \subset \mathscr{C}$, 
where $\mathscr{C}$ is the fibered cone of  $T_{\widetilde{\beta}}$ containing 
$ [F_{\widetilde{\beta}}] \in  C_{\widetilde{\beta}} $.

\begin{ex}
\label{ex_demonstration}
Consider the braid $\beta = \sigma_3^2 \sigma_4 \in B_5$. 
The braid $\beta$ is increasing in the middle. 
The skew-palindromization $\widetilde{\beta} = \sigma_1 \sigma_2^2 \sigma_3^2 \sigma_4 \in B_5$  is pseudo-Anosov, 
see the proof of Step 1 in \cite[Proof of Theorem D]{HiroseKin201}. 
See Figure \ref{fig_subcone_ver2}(2) for the braided link. 
\end{ex}

\begin{thm}
\label{thm_fibered-cone}
Let $\beta \in B_{2n+1}$ be a braid increasing in the middle. 
Suppose that the skew-palindromization $\widetilde{\beta}$ of $\beta$ is pseudo-Anosov. 
Let $(x,y) \in C_{\widetilde{\beta}} $ be a primitive integral class. 
Then we have the following. 
\begin{enumerate}
\item[(1)]
There exists a braid $\alpha_{(x,y)} \in B_{\|(x,y)\|+1}$ increasing in the middle 
such that 
the monodromy $\phi_{(x,y)}: F_{(x,y)} \rightarrow F_{(x,y)}$ 
of the fibration $T_{\widetilde{\beta}} \rightarrow S^1$ 
associated with $(x,y)$ 
is given by the skew-palindromization $\widetilde{\alpha_{(x,y)}}$ of $\alpha_{(x,y)}$. 
(In particular $\widetilde{\alpha_{(x,y)}}$ is a pseudo-Anosov braid.)

\item[(2)] 
Let $\alpha_{(x,y)}^{\bullet} \in B_{\|(x,y)\|}$ be the braid obtained from $\alpha_{(x,y)}$ 
by removing the strand of the middle index. 
Then $C(\beta^{\bullet})= C(\alpha_{(x,y)}^{\bullet})$, and hence 
$\mathfrak{t}(\widetilde{\alpha_{(x,y)}^{\bullet}}) \in \mathcal{D}_{\frac{\|(x,y)\|}{2}-1}(M_{C(\beta^{\bullet})})$. 

\end{enumerate}
\end{thm}

For the proof of  Theorem \ref{thm_fibered-cone}, 
we need some preparations from \cite{HiroseKin201}. 
Let $L$ be a link in $S^3$. 
Suppose that an unknot  $K$ is a component of $L$. 
Then the exterior $\mathcal{E}(K)$ is a solid torus 
(resp. the boundary of the exterior $\partial \mathcal{E}(K)$ is a torus). 
We take a disk $D$ bounded by the longitude of a tubular neighborhood  $ \mathcal{N}(K)$ of $K$. 
We define a mapping class $t_D$ defined on $\mathcal{E}(K)$ 
as follows. 
We cut $ \mathcal{E}(K)$ along $D$. 
 We have  resulting two sides obtained from $D$, and reglue two sides by twisting $360$ degrees 
so that the mapping class defined on the torus $\partial  \mathcal{E}(K)$ 
is the right-handed Dehn twist about $\partial D$. 
We call such a mapping class $t_D$ on $ \mathcal{E}(K)$  the {\it disk twist about} $D$. 
For simplicity, we also call a representative of the mapping class $t_D$ 
the  {\it disk twist about}  $D$, 
and denote it by the same notation 
$$t_D:  \mathcal{E}(K) \rightarrow  \mathcal{E}(K) .$$

For any integer $\ell $, consider the homeomorphism 
$$ t_D^{\ell}:  \mathcal{E}(K) \rightarrow   \mathcal{E}(K).$$
Observe that $ t_D^{\ell}$  converts the link $L $ into a link 
 $K \cup t_D^{\ell}(L- K)$ 
 so that $S^3 \setminus L$ is homeomorphic to $S^3 \setminus (K \cup t_D^{\ell}(L- K))$. 
 Then $ t_D^{\ell}$ induces a homeomorphism $h_{D,\ell}$ 
 between the exteriors of links 
 $L$ and $K \cup t_D^{\ell}(L-K)$: 
 $$h_{D,\ell}: \mathcal{E}(L) \rightarrow \mathcal{E}(K \cup t_D^{\ell}(L-K)).$$

Consider the braided link $L= \mathrm{br}(b)= A \cup \mathrm{cl}(b)$ for a braid $b$ with the braid axis $A$. 
We consider the 
$k$-th power of the disk twist  about  the disk $D_A$ 
bounded by the longitude of $\mathcal{N}(A)$: 
$$t_{D_A}^{k}: \mathcal{E}(A) \rightarrow  \mathcal{E}(A).$$ 
Note that 
$A \cup t_{D_A}^k(\mathrm{cl}(b)) = A \cup \mathrm{cl}(b \Delta^{2k}) = \mathrm{br}(b \Delta^{2k})$. 
Hence $h_{D_A, k}$ sends 
$\mathcal{E}(\mathrm{br}(b))= \mathcal{E}(A \cup \mathrm{cl}(b))$ to 
$\mathcal{E}(\mathrm{br}(b \Delta^{2k})) = \mathcal{E}(A \cup \mathrm{cl}(b \Delta^{2k}))$.

Following \cite[Section 4.1]{HiroseKin201}, we next introduce a sequence of braided links $\{\mathrm{br}(b_p)\}_{p=1}^{\infty}$ 
obtained from an  $i$-increasing braid $b \in B_n$ 
such that  $T_{b_p} \simeq T_b$  
i.e., the mapping tori $T_{b_p}$ and $T_b$ are homeomorphic to each other for each $p \ge 1$. 
We set $u= u(b,i)$ 
that is the intersection number of the pair $(b,i)$. 
Let $D$ be an associated disk of the pair $(b,i)$. 
We take a disk twist 
$$t_D: \mathcal{E}(\mathrm{cl}(b(i))) \rightarrow \mathcal{E}(\mathrm{cl}(b(i)))$$  
so that the point of intersection $D \cap A$ becomes the center of the twisting about $D$, i.e.,  
$t_D(D \cap A) = D \cap A$. 
It follows that 
$$t_D(\mathrm{br}(b - b(i))) \cup \mathrm{cl}(b(i))$$ 
is a braided link 
of a $j$-increasing braid for some index $j$ with $(n+u)$ strands. 
(cf. Figures 11 and 12 in \cite{HiroseKin201}.) 
We define $b_1 $ to be such a braid with $(n+u)$ strands. 
The trivial knot $t_D(A)(=A)$ becomes a braid axis of $b_1$. 
By definition of the disk twist, we have $T_{b_1} \simeq T_b$. 
We remark that there is some ambiguity in defining  $b_1$. 
However the braid $b_1$ is well defined up to conjugate, see \cite[Section 4.1]{HiroseKin201}. 
The conjugacy class of $b_1$ is denoted by $\langle b_1 \rangle$.

To define the braid $b_p$ obtained from the above $b$ for $p  \ge 1$, we consider the $p$-th power 
$$t_D^{ p}: \mathcal{E}(\mathrm{cl}(b(i))) \rightarrow \mathcal{E}(\mathrm{cl}(b(i)))$$
using the above disk twist $t_D$. 
As in the case of $p=1$, 
$$t_D^{p}(\mathrm{br}(b - b(i))) \cup \mathrm{cl}(b(i))$$ 
is a braided link of an increasing braid for some index with $(n+pu)$ strands. 
We define $b_p \in B_{n+pu}$ to be such a braid with $n+pu$ strands.  
Then  $T_{b_p} \simeq T_b$. 
As in the case of $p=1$, the braid $b_p$ is well defined up to conjugate. 
We denote by $\langle b_p \rangle$, the conjugacy class of such a  braid $b_p$.   
We say that $\langle b_p \rangle$ (or a representative $b_p$) is {\it obtained from $b$ by the disk twist $t_D$} {\it ($p$ times)}.

Now we suppose that a braid $b$ is of the form $b= \widetilde{\beta}$, 
where $\beta \in B_{ 2n+1}$ is a braid increasing in the middle. 
Then $\widetilde{\beta}$ is also increasing in the middle. 
The following lemma describes a property of a representative 
of the conjugacy class $\langle (\widetilde{\beta})_p \rangle$ 
obtained from $\widetilde{\beta}$ by the disk twist $p$ times.

\begin{lem}
\label{lem_disktwist-skew}
Let $\beta \in B_{2n+1}$ be a braid increasing in the middle with the intersection number 
$u= u(\beta, n+1)$. 
We consider the skew-palindromization $\widetilde{\beta}$ (that is  increasing in the middle).   
Let $\langle (\widetilde{\beta})_p \rangle$ be the conjugacy class of a braid 
 obtained from $\widetilde{\beta}$  by the disk twist $p$ times for $p \ge 1$. 
We have the following. 
\begin{enumerate}
\item[(1)] 
There exists a braid $\alpha= \alpha(p) \in B_{2n+ 2pu+1}$ increasing in the middle 
such that  
$\widetilde{\alpha}=  \widetilde{\alpha(p)} \in \langle (\widetilde{\beta})_p \rangle $, 
i.e.,  
$\widetilde{\alpha}=  \widetilde{\alpha(p)}$ represents the conjugacy class $ \langle (\widetilde{\beta})_p \rangle$.

\item[(2)] 
$C(\beta^{\bullet})= C(\alpha^{\bullet})$, 
where $\alpha^{\bullet}= \alpha(p)^{\bullet} \in B_{2n+ 2pu}$.

\item[(3)] 
$\mathfrak{t}(\widetilde{\alpha^{\bullet}})  \in \mathcal{D}_g(M_{C(\beta^{\bullet})})$, 
where $g= n+pu-1$. 
\end{enumerate}
\end{lem}

\begin{center}
\begin{figure}
\includegraphics[height=5in]{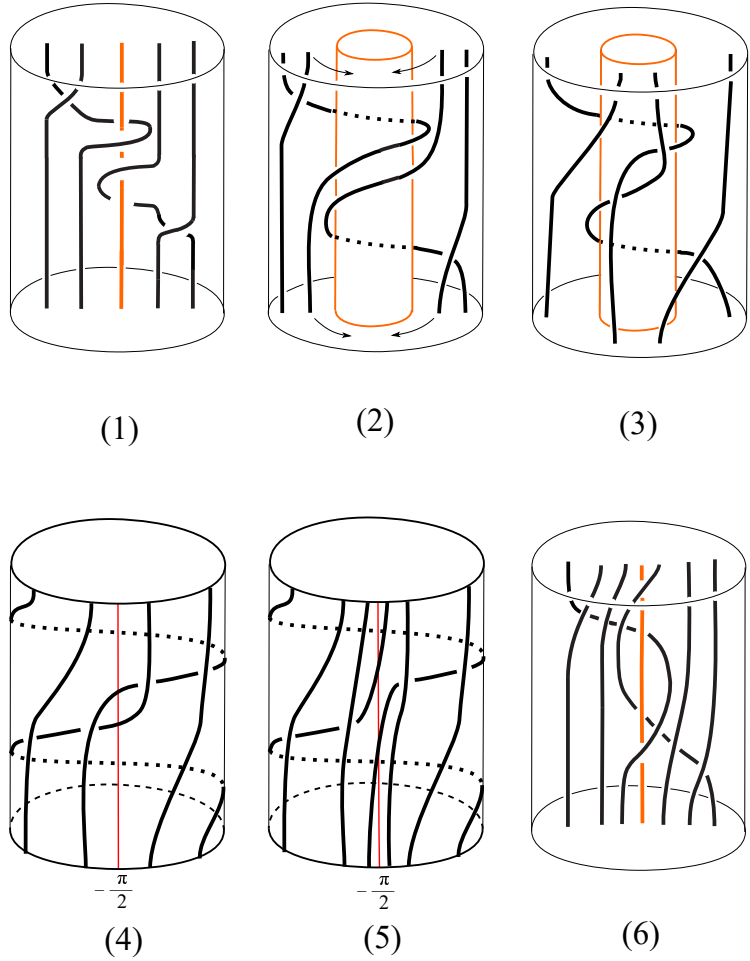} 
\caption{(1) The skew-palindromization 
$\widetilde{\beta}$ of $\beta=\sigma_3^2 \sigma_4 \in B_5$. 
(2) The regular neighborhood of the middle strand (third strand) is removed.  
(3) The base points are moved to the circle 
$\{ \frac{1}{2} e^{i \theta} \, | \, 0 \leq \theta \leq 2 \pi \}$. 
(4) Project $\mathrm{cl}(\widetilde{\beta^{\bullet}})$ to the torus. 
(5) Dehn twist about the circle $c= \{ -\frac{\pi}{2} \} \times [0,1]$ on the torus. 
(6) The skew-palindromization $\widetilde{\alpha(1)}$ 
of $\alpha(1)= \sigma_3 \sigma_4 \sigma_5 \sigma_6 \sigma_3 \in B_7$.  
(cf. Figure 7 in  \cite{HiroseKin202}.)
}
\label{fig:example}
\end{figure}
\end{center}

\begin{center}
\begin{figure}
\includegraphics[height=2in]{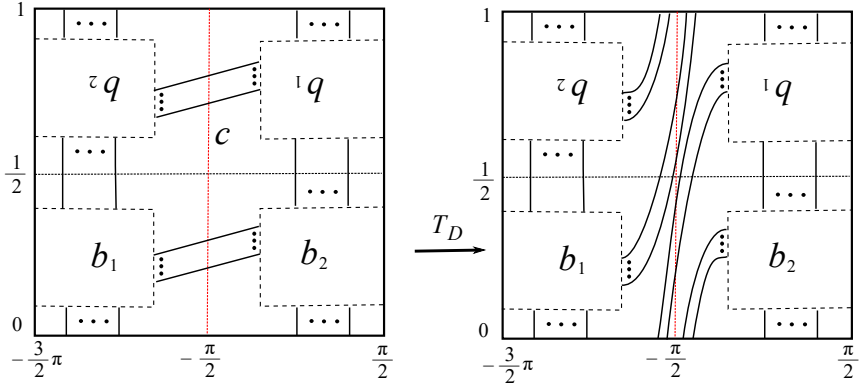} 
\caption{The braid $\beta^{\bullet} \in B_{2n}$ consists of $b_{1}$, $b_{2}$ and slanted parallel 
arcs between them. 
The disk twist $t_{D}$ induces a Dehn twist about 
$c= \{ -\frac{\pi}{2} \} \times [0,1]$ on the torus.  
Compare Figure \ref{fig:example}(4)(5) with Figure \ref{fig:torus-projection}.}
\label{fig:torus-projection}
\end{figure}
\end{center}

\begin{proof}
(1) 
Figure \ref{fig:example}  illustrates the procedure of the proof. 
(See also Example \ref{ex_demonstration}.)
We put $\mathrm{cl}(\widetilde{\beta})$ in $D^{2} \times S^{1}$ $= D^{2} \times ([0,1]/0\sim 1)$ 
such that $\mathrm{cl}(\widetilde{\beta}) \cap D^{2} \times [0,1/2]$ $= \beta$, 
$\mathrm{cl}(\widetilde{\beta}) \cap D^{2} \times [1,1/2]$ $= \mathrm{skew}(\beta)$
and $\mathrm{cl}(\widetilde{\beta}(n+1))$ corresponds to $(\text{the center of }D^{2}) \times S^{1}$, 
as shown in (1) of Figure \ref{fig:example}. 
Let $R_{D^{2}\times S^{1}}:D^{2}\times S^{1} \to D^{2}\times S^{1}$ be the involution 
induced by the involution $R(r e^{i\theta},t) = (r e^{i(\pi-\theta)},1-t)$ on $D^{2} \times [0,1]$. 
Let $\mathcal{N}$ be the tubular neighborhood of $(\text{the center of }D^{2}) \times S^{1}$ 
such that $\mathrm{cl}(\widetilde{\beta}) \cap (D^{2} \times S^{1} \setminus \mathrm{int} \, \mathcal{N})$ 
$=$ $\mathrm{cl}(\widetilde{{\beta}^{\bullet}})$, 
as shown in (2) of Figure \ref{fig:example}. 
We identify $D^{2} \times S^{1} \setminus \mathrm{int} \, \mathcal{N}$ with 
$[0,1] \times S^{1} \times S^{1} 
= [0,1] \times ([0, 2\pi]/ 0 \sim 2 \pi) \times ([0,1]/0 \sim 1)$ 
so that $\partial \mathcal{N} = \{ 0 \} \times S^{1} \times S^{1}$, 
$\{ \frac{1}{2} e^{i \theta} \, | \, 0 \leq \theta \leq 2 \pi \} 
 \times \{ \mathrm{pt} \} 
= \{ \frac{1}{2} \} \times S^{1}
 \times \{ \mathrm{pt} \} $, 
and for the disk $D$ associated to the pair $(\widetilde{\beta}, n+1)$, 
we have 
$(D^{2} \times S^{1} \setminus \mathrm{int} \, \mathcal{N}) \cap D = [0,1] \times \{ -\frac{\pi}{2}  \} \times S^{1}$. 
We deform $\mathrm{cl}(\widetilde{\beta^{\bullet}})$ 
so as to intersect $D^2 \times \{ 0 \}$ in $2n$ points 
$\{ \frac{1}{2} e^{i \theta} | \theta= -\frac{1}{2n+2}\pi, -\frac{2}{2n+2}\pi, \ldots, 
-\frac{n}{2n+2}\pi, -\frac{n+2}{2n+2}\pi, -\frac{n+3}{2n+2}\pi, \ldots, -\frac{2n+1}{2n+2}\pi \}$ 
and not to occur new intersections with $[0,1] \times \{ -\frac{\pi}{2} \} \times S^1$ 
by the isotopy commuting with the involution $R$, 
as shown in (3) of Figure \ref{fig:example}. 
We make a projection $\mathcal{D}(\mathrm{cl}(\widetilde{\beta^{\bullet}}))$ 
of $\mathrm{cl}(\widetilde{\beta^{\bullet}})$ onto 
$\{ \frac{1}{2} \} \times S^1 \times S^1$ with at most double points, and 
each double point indicates which is over pass or under pass 
in the same way as a knot diagram, 
as shown in (4) of Figure \ref{fig:example}. 
For short, we drop $\{ \frac{1}{2} \}$ from $\{ \frac{1}{2} \} \times S^1 \times S^1$ 
and identify it with a torus 
$S^{1} \times S^{1} = 
([-\frac{3}{2}\pi, \frac{1}{2}\pi]/ -\frac{3}{2}\pi \sim \frac{1}{2} \pi) 
\times ([0,1]/0 \sim 1)$. 
With this identification, the restriction of $R$ on $\{ \frac{1}{2} \} \times S^1 \times S^1$ 
is an involution which maps $(\theta, t)$ to $(-\pi - \theta, 1-t)$, i.e.,  
a $\pi$-rotation about $(-\frac{\pi}{2}, \frac{1}{2})$. 
Since $\beta$ is an increasing braid in the middle, $\mathcal{D}(\mathrm{cl}(\widetilde{\beta^{\bullet}}))$ 
intersects $-\frac{\pi}{2} \times [\frac{1}{2},1]$ in $u= u(\beta, n+1)$ points 
and $-\frac{\pi}{2} \times [0, \frac{1}{2}]$ in $u= u(\beta,n+1)$ points as shown in 
Figure \ref{fig:torus-projection}. 
The disk twist $t_{D}$ induces the Dehn twist about the circle 
$c=\{ -\frac{\pi}{2} \} \times [0,1]$ in $S^{1} \times S^{1}$ which commutes with 
$R|_{S^{1} \times S^{1}}$. 
This Dehn twist changes $\mathcal{D}(\mathrm{cl}(\widetilde{\beta^{\bullet}}))$ 
as shown in the right of Figure \ref{fig:torus-projection} 
and (5) of  Figure \ref{fig:example}. 
Now we append $(\text{the center of }D^{2}) \times S^{1}$ to the above result, 
that is, append an under-going string $\{ -\frac{\pi}{2} \} \times [0,1]$ in 
the right of Figure \ref{fig:torus-projection} and regard the new diagram 
as  a closure of the planar braid, 
as shown in (6) of Figure \ref{fig:example}. 
Let $\alpha(p)$ be the braid indicated by the lower part of the result. 
Then the upper part corresponds to $\mathrm{skew}(\alpha(p))$, and hence 
$ \widetilde{\alpha(p)}$ is a representative of $\langle \widetilde{\beta}_{p}\rangle $. 
The proof of (1) is done.


(2) 
Let $b \in B_{2g+2}$ and $f: S^1 \times S^1 \rightarrow S^1 \times S^1$ be the same notations as in \cite[p.1811]{HiroseKin202}. 
Let $\gamma$ (resp. $b_f$) be a braid on the annulus in the assumption $(*)$ 
(resp. the assumption $(**)$) in \cite[p.1811]{HiroseKin202}. 
We set $b:=  \beta^{\bullet}$ and $f:= (t_c)^p$. 
Then $b_f = \alpha(p)^{\bullet} (= \alpha^{\bullet})$ and 
$\gamma=\widetilde{\alpha(p)^{\bullet}} (= \widetilde{\alpha^{\bullet}})$ 
satisfy the assumptions  $(*)$ and $(**)$ in \cite[p.1811]{HiroseKin202}. 
By Lemma 4 of \cite{HiroseKin202}, we conclude that 
$C(\beta^{\bullet})$ and $C(\alpha^{\bullet})$ are ambient isotopic. 
This completes the proof of (2).

(3) The claim (3) holds from the claim (2) and Theorem  \ref{thm_2-fold}. 
This completes the proof of Lemma  \ref{lem_disktwist-skew}. 
 \end{proof}

\begin{remark}
The proof of Lemma \ref{lem_disktwist-skew} 
tells us that one can describe the braid $\alpha(p)$ in Lemma \ref{lem_disktwist-skew} concretely. 
In fact it is possible to read off $\alpha(p)$ from Figure \ref{fig:torus-projection}. 
For example, in the case $\beta = \sigma_3^2 \sigma_4 \in B_5$ as in Figure \ref{fig:example}(1), 
the representative $ \widetilde{\alpha(p)} \in \langle (\widetilde{\beta})_p \rangle$ 
is the skew-palindromization of the braid 
$\alpha(p)= \sigma_3 \sigma_4 \dots \sigma_{4+ 2p} \sigma_{2+p} \in B_{5+ 2p}$ 
that is increasing in the middle. 
\end{remark}

Let $\beta \in B_{2n+1}$ be the  braid increasing in the middle with the intersection number 
$u= u(\beta, n+1)$ as before. 
By Lemma \ref{lem_product} and Example \ref{ex_halftwist}, 
the product $\beta \Delta^k \in B_{2n+1} $ of $\beta$ and $\Delta^k$ for $k \ge 1$ 
is also increasing in the middle 
such that 
$$u(\beta \Delta^k, n+1) = u(\beta, n+1) + u(\Delta^k, n+1) = u+ kn.$$
Consider the skew-palindromization $\widetilde{\beta \Delta^k}$ of $\beta \Delta^k$. 
Recall that $\mathrm{skew}(\Delta)= \Delta$ (see Section \ref{subsection_homogeneous}), 
and hence $ \mathrm{skew}( \Delta^k) = \Delta^k$. 
Thus it follows that 
$$\widetilde{\beta \Delta^k} = \mathrm{skew}(\beta \Delta^k) \beta \Delta^k = 
\mathrm{skew}(\Delta^k)  \mathrm{skew}(\beta) \beta \Delta^k = \Delta^k \widetilde{\beta} \Delta^k.$$ 
By Examples \ref{ex_product} and \ref{ex_halftwist}, 
$\widetilde{\beta \Delta^k}$ is a braid increasing in the middle  
with  $u(\widetilde{\beta \Delta^k}, n+1) = 2 u(\beta \Delta^k, n+1)= 2u+ 2kn.$
Since $\Delta^k \widetilde{\beta} \Delta^k$ is conjugate to $\widetilde{\beta} \Delta^{2k}$ in $B_{2n+1}$, 
we have 
$$\mathrm{br}(\widetilde{\beta \Delta^k}) = \mathrm{br}(\Delta^k \widetilde{\beta} \Delta^k) 
= \mathrm{br}(\widetilde{\beta} \Delta^{2k}). $$

\begin{lem}
\label{lem_disktwist-skew2}
Let $\beta \in B_{2n+1}$ be a braid increasing in  the middle  with the intersection number $u= u(\beta, n+1)$. 
For $k \ge 1$,  we consider $\beta \Delta^k \in B_{2n+1}$ and its skew-palindromization $\widetilde{\beta \Delta^k}$. 
Let  $ \langle (\widetilde{\beta \Delta^k})_p \rangle$ 
be the conjugacy class of a braid obtained from $\widetilde{\beta \Delta^k}$ by the disk twist $p$ times 
for $p \ge 1$. 
We have the following. 
\begin{enumerate}
\item[(1)] 
There exists a braid $\gamma= \gamma(k, p) \in B_{2n+1+ p(2u+ 2kn)}$ increasing in the middle 
such that $\widetilde{\gamma} = \widetilde{\gamma(k,p)} \in \langle (\widetilde{\beta \Delta^k})_p \rangle$. 

\item[(2)] 
$C(\beta^{\bullet})=  C(\gamma^{\bullet}) $, 
where $\gamma^{\bullet}= \gamma(k,p)^{\bullet} \in B_{2n+ p(2u+ 2kn)}$.

\item[(3)] 
$\mathfrak{t}(\widetilde{\gamma^{\bullet}})  \in \mathcal{D}_g(M_{C(\beta^{\bullet})})$ 
for $g= n+pu+ pkn-1 \equiv pu-1 \pmod n$. 
\end{enumerate}
\end{lem}

\begin{proof}
Recall that for each $k \ge1$, 
$\beta \Delta^k \in B_{2n+1} $ is a braid increasing  in the middle with the intersection number $u(\beta \Delta^k, n+1) = u+ kn$. 
The claim (1) follows immediately from Lemma \ref{lem_disktwist-skew}(1).  
For the proof of the claim (2), note that 
$$(\beta \Delta^k)^{\bullet} = \beta^{\bullet} (\Delta^k)^{\bullet} = \beta^{\bullet} \Delta_{2n}^k \in B_{2n}, $$ 
see Example \ref{ex_halftwist}. 
Thus 
$C((\beta \Delta^k)^{\bullet}) = C(\beta^{\bullet} \Delta_{2n}^k) = C(\beta^{\bullet} )$, 
see (\ref{equation_halftwist}) in Section \ref{subsection_closures} 
for the  second equality.  
By Lemma \ref{lem_disktwist-skew}(2), 
we have $C((\beta \Delta^k)^{\bullet}) = C(\gamma^{\bullet})$.  
Putting them together, we obtain $C(\beta^{\bullet})= C(\gamma^{\bullet})$. 
The proof of (2) is done. 
The claim (3) holds from the claim (2) and Theorem  \ref{thm_2-fold}. 
This completes the proof. 
\end{proof}

We are now ready to prove Theorem \ref{thm_fibered-cone}. 

\begin{proof}[Proof of Theorem \ref{thm_fibered-cone}]
By the assumption of Theorem \ref{thm_fibered-cone}, 
$\widetilde{\beta}$ is pseudo-Anosov, and it is increasing in the middle. 
By Theorem \ref{thm_disktwist}(1)(2), 
the subcone $C_{\widetilde{\beta}}= C_{(\widetilde{\beta}, n+1)}$ 
is a subset of  the fibered cone $\mathscr{C}$ containing $[F_{\widetilde{\beta}}]$. 
Moreover the fiber $F_{(x,y)}$ for each primitive integral class $(x,y) \in C_{\widetilde{\beta}}$ 
has genus $0$.

Let $D= D_{(\widetilde{\beta}, n+1)}$ be the associated disk 
of  the  braid $\widetilde{\beta}$ increasing in the middle.  
We consider two types of the disk twists. 
One is $t_{D_A}^k: \mathcal{E}(A) \rightarrow \mathcal{E}(A)$ 
for the braid axis $A$ of $\widetilde{\beta}$,  
and the other is 
$t_D^p: \mathcal{E}(\mathrm{cl}(\widetilde{\beta}(n+1))) \rightarrow \mathcal{E}(\mathrm{cl}(\widetilde{\beta}(n+1))) $, 
where 
$\widetilde{\beta}(n+1)$ is the middle strand  of the $(2n+1)$-braid $\widetilde{\beta}$. 
Consider the homeomorphisms 
\begin{eqnarray*}
h_{D_{A}, k}&:& \mathcal{E}(\mathrm{br}(\widetilde{\beta})) \rightarrow \mathcal{E}(\mathrm{br}(\widetilde{\beta} \Delta^{2k})) 
=  \mathcal{E}(\mathrm{br}(\widetilde{\beta \Delta^k} )), 
\\
h_{D, p}&:& \mathcal{E}(\mathrm{br}(\widetilde{\beta})) \rightarrow  \mathcal{E}(\mathrm{br}((\widetilde{\beta})_p )) 
\simeq \mathcal{E} (\mathrm{br}(\widetilde{\alpha(p)})), 
\end{eqnarray*}
where 
$\widetilde{\alpha(p)}$ is the braid obtained from Lemma \ref{lem_disktwist-skew}(1). 
We obtain the skew-palindromization  
$\widetilde{\beta \Delta^k} = \Delta^k \widetilde{\beta} \Delta^k$ (that is increasing in  the middle)  
from the former homeomorphism $h_{D_{A}, k}$. 
We also obtain the skew-palindromization   
$\widetilde{\alpha(p)} $ (that is increasing in  the middle)  
from the latter  homeomorphism $h_{D, p}$. 
Both  braids are pseudo-Anosov, since 
the exteriors of the links $\mathrm{br}(\widetilde{\beta})$, $\mathrm{br}(\widetilde{\beta \Delta^{k}})$ 
and $\mathrm{br}(\widetilde{\alpha(p)})$ are homeomorphic to each other. 
Hence one can further apply two types of the disk twists for each of the two braids $\widetilde{\beta \Delta^k}$ and $\widetilde{\alpha(p)}$. 
Then the resulting braids are again 
 the skew-palindromization  of some braids that are increasing in the middle  by 
 Lemmas \ref{lem_disktwist-skew}(1) and \ref{lem_disktwist-skew2}(1). 
Choosing  two types of the disk twists alternatively, 
one obtains a family of skew-paindromizations (of some braids) that are  increasing in the middle.  
By the proof of Theorem 3.2(3) in \cite{HiroseKin201}, 
the monodromy $\phi_{(x,y)}: F_{(x,y)} \rightarrow F_{(x,y)}$ 
of the fibration  $T_{\widetilde{\beta}} \rightarrow S^1$ 
associated with any primitive integral class  $ (x,y) \in C_{\widetilde{\beta}}$ 
is given by a braid, 
say the skew-palindromization $\widetilde{\alpha_{(x,y)}}$ of some braid $\alpha_{(x,y)}$ in the family. 
The planar braid $ \alpha_{(x,y)}$ is the desired braid. 
Let $2N+1$ be the number of the strands of $ \alpha_{(x,y)}$ that is increasing in the middle. 
Since the Thurston norm $\|(x,y)\|$ of the class $(x,y)$ 
is the negative Euler characteristic of the $(2N+1)$-punctured disk 
that is equal to $2N$. 
Thus $2N+1= \|(x,y)\|+1$ and hence $\alpha_{(x,y)} \in B_{\|(x,y)\|+1}$. 
This completes the proof of (1). 

 The claim (2)  follows from Lemma \ref{lem_disktwist-skew}(2)(3) and Lemma \ref{lem_disktwist-skew2}(2)(3)
 together with the above argument in the proof of (1). 
 \end{proof}

\section{Applications}
\label{section_applications}

For the proofs of Theorems \ref{thm_infinite} and \ref{thm_infinite-hyperbolic}, 
we first prove the following result. 

\begin{prop}
\label{prop_finite} 
Let $L$ be a link in $S^3$. 
Let 
$\beta_{(1)},  \dots, \beta_{(n)} \in B_{2n+1}$ be increasing in the middle for some $n \ge 2$. 
Suppose that $\beta_{(1)},  \dots, \beta_{(n)}$ satisfy the following conditions (1)--(3): 
For each $j= 1, \dots, n$,
\begin{enumerate}
\item[(1)] 
$u(\beta_{(j)}, n+1) \equiv j$ $\pmod n$, 
where $u(\beta_{(j)}, n+1)$ is the intersection number of the pair $(\beta_{(j)}, n+1)$.

\item[(2)] 
$L = C(\beta_{(j)}^{\bullet})$, where $\beta_{(j)}^{\bullet} \in B_{2n}$. 

\item[(3)] 
The skew-palindromization $\widetilde{\beta_{(j)}} $ of $\beta_{(j)}$ is  pseudo-Anosov. 
\end{enumerate}
Then we have $\ell_g(M_L) \asymp \dfrac{1}{g}$.  
\end{prop}

\begin{proof}
We fix $j \in \{1, \dots, n\}$ for a moment, and apply Lemma \ref{lem_disktwist-skew2} 
to $ \beta_{(j)}$, $k \ge 1$ and $p=1$. 
Let $ \gamma(k,1)$ be a braid increasing in  the middle 
given by Lemma \ref{lem_disktwist-skew2}(1). 
By \cite[Theorem 5.2]{HiroseKin201}, 
a representative of $ \langle (\widetilde{\beta_{(j)} \Delta^k})_1 \rangle$ gives the monodromy 
$\phi_{(k+1,1)}: F_{(k+1, 1)} \to F_{(k+1,1)}$ corresponding to the primitive integral class 
$(k+1,1) \in C_{\widetilde{\beta_{(j)}}}$ of the fibered $3$-manifold $T_{\widetilde{\beta_{(j)}}}$. 
See Figure \ref{fig_subcone_ver2}(1) for the class $(k+1,1)$. 
In particular 
the representative $\widetilde{\gamma(k,1)} \in  \langle (\widetilde{\beta_{(j)} \Delta^k})_1 \rangle$ 
gives the monodromy $\phi_{(k+1,1)}: F_{(k+1, 1)} \to F_{(k+1,1)}$ 
of the fibration $T_{\widetilde{\beta_{(j)}}} \rightarrow S^1$.  
Hence we can say that 
$\widetilde{\gamma(k,1)}$ is a braid with $\|(k+1,1)\|+1$ strands. 
(Recall that $\|(x,y)\|$ is the Thurston norm of the class $(x,y)$.)

Note that the ray of the class $(k+1,1) = (k+1)(1, \frac{1}{k+1})$ through the origin converges to the ray of $(1,0)$ 
as $k \to \infty$.  
This together with Theorem \ref{thm_Fried_1}(2) implies that 
$$\mathrm{Ent}(\widetilde{\gamma(k,1)}) = \mathrm{Ent}((k+1, 1))= 
\mathrm{Ent}((1, \tfrac{1}{k+1})) \rightarrow \mathrm{Ent}((1,0)) \ \mbox{as}\ k \to \infty.$$
Since the monodromy on the fiber $F_{(1,0)} = F_{\widetilde{\beta_{(j)}}}$  
is given by  $\widetilde{\beta_{(j)}}$, 
\begin{equation}
\label{equation_converge}
\mathrm{Ent}(\widetilde{\gamma(k,1)})  =  \mathrm{Ent}((k+1, 1))
\rightarrow \mathrm{Ent}((1,0)) = 
 \mathrm{Ent}(\widetilde{\beta_{(j)}}) \ \mbox{as}\ k \to \infty.
 \end{equation}

By \cite[Lemma 6.3]{HiroseKin201}, for $k$ large, 
$\widetilde{\gamma(k,1)^{\bullet}} \in B_{\|(k+1,1)\|}$ is pseudo-Anosov with 
the same dilatation as $\widetilde{\gamma(k,1)}$. 
By the arguments in the proof of  \cite[Lemma 6.3]{HiroseKin201}, 
one sees that  for $k$ large, the pseudo-Anosov braid $\widetilde{\gamma(k,1)^{\bullet}}$ 
satisfies the condition $\diamondsuit$ in Lemma \ref{lem_disk}. 
Therefore, 
for $k$ large, 
$\mathfrak{t}(\widetilde{\gamma(k,1)^{\bullet}}) $ is still pseudo-Anosov 
with the same dilatation as $\widetilde{\gamma(k,1)^{\bullet}}$. 
Then by Theorem \ref{thm_2-fold}, it holds 
$\mathfrak{t}(\widetilde{\gamma(k,1)^{\bullet}}) \in \mathcal{D}_{\frac{\|(k+1,1)\|}{2}-1} (M_L)$, 
where 
$L= C(\beta_{(j)}^{\bullet}) = C(\gamma(k,1)^{\bullet})$ 
by Lemma \ref{lem_disktwist-skew2}(2). 
Putting them together, we have
$$\lambda(\mathfrak{t}(\widetilde{\gamma(k,1)^{\bullet}}))= \lambda(\widetilde{\gamma(k,1)^{\bullet}}) =\lambda( \widetilde{\gamma(k,1)}) 
= \lambda((k+1,1)),$$
where $\lambda((x,y))$ denotes the dilatation of the class $(x,y)$, 
i.e., 
$\log(\lambda((x,y))) = \mathrm{ent}((x,y))$.  
(See Theorem \ref{thm_Fried_1}(1).)
Since $\mathfrak{t}(\widetilde{\gamma(k,1)^{\bullet}})$ is the mapping class on the closed surface of 
genus $\frac{\|(k+1,1)\|}{2}-1$, we have 
\begin{equation}
\label{equation_bullet}
  \begin{split}
    \mathrm{Ent}(\mathfrak{t}(\widetilde{\gamma(k,1)^{\bullet}}))  &=(\|(k+1,1)\| -4) \log (\lambda( \mathfrak{t}(\widetilde{\gamma(k,1)^{\bullet}})))  \\
            &= (\|(k+1,1)\| -4) \mathrm{ent}((k+1,1)).
  \end{split}
\end{equation}


\noindent
{\bf Claim 1}. 
$\mathrm{ent}((k+1,1)) \rightarrow 0$ as $k \rightarrow \infty$. 
\medskip

\noindent
Proof of Claim 1. 
By (\ref{equation_converge}), 
$\mathrm{Ent}((k+1, 1)) (= \|(k+1,1)\| \mathrm{ent}((k+1,1))) \rightarrow  \mathrm{Ent}((1,0)) $ 
as $k \to \infty$. 
This implies that there exists a constant $P>0$ independent of $k$ such that 
$$0 < \|(k+1,1)\| \mathrm{ent}((k+1,1)) < P$$ 
for all $k \ge 1$. 
Since $ \|(k+1,1)\| \to \infty$ as $k \to \infty$, 
we obtain 
$$\mathrm{ent}((k+1,1)) < \frac{P}{\|(k+1,1)\|} \rightarrow 0 \hspace{2mm}\mbox{as}\ k \to \infty.$$
This completes the proof of Claim 1. 
\medskip


By  (\ref{equation_bullet}), Claim 1 and (\ref{equation_converge}), one has 
\begin{equation}
\label{equation_converge-t}
 \begin{split}
  \lim_{k \to \infty}\mathrm{Ent}(\mathfrak{t}(\widetilde{\gamma(k,1)^{\bullet}}))  
  &=\lim_{k \to \infty} (\|(k+1,1)\| -4) \mathrm{ent}((k+1,1))  
\\
&= \lim_{k \to \infty} \|(k+1,1)\| \mathrm{ent}((k+1,1)) 
\\
&= \lim_{k \to \infty} \mathrm{Ent}((k+1,1)) \hspace{2mm} (\because\mbox{definition of } \mathrm{Ent}(\cdot))
\\
&= \mathrm{Ent}(\widetilde{\beta_{(j)}}).
  \end{split}
\end{equation}
%
%
%
%
%
On the other hand,  Lemma \ref{lem_disktwist-skew2}(3)  tells us that 
$\mathfrak{t}(\widetilde{\gamma(k,1)^{\bullet}}) \in \mathcal{D}_g(M_L)$, 
where 
\begin{equation}
\label{equation_genus}
g = n + u(\beta_{(j)}, n+1) + kn -1 \equiv  u(\beta_{(j)}, n+1)-1 \equiv j-1 \pmod n.
\end{equation}
(See the condition (1) of Proposition \ref{prop_finite}.) 
For $k \ge 1$, consider the set of all pseudo-Anosov mapping classes 
$\mathfrak{t}(\widetilde{\gamma(k,1)^{\bullet}})$ obtained from 
 $\beta_{(j)}$ over all $j= 1, \dots, n$. 
Then by (\ref{equation_genus}) together with  the condition (1) of Proposition \ref{prop_finite}, 
one can find 
a sequence of pseudo-Anosov elements 
$\phi_g \in \mathcal{D}_g(M_L)$ for $g \gg 0$ 
in this set. 
 In fact when $g \equiv j-1$  $\pmod n$, 
 one can put  $\phi_g= \mathfrak{t}(\widetilde{\gamma(k,1)^{\bullet}})$ 
 obtained from the braid  $\beta_{(j)}$, 
 where $k$ satisfies the equality (\ref{equation_genus}). 
Since each of braids 
$\widetilde{\beta_{(1)}}, \dots, \widetilde{\beta_{(n)}}$ satisfies (\ref{equation_converge-t}), 
there exists a constant $C'>0$ independent of $g$  
so that 
$\ell_g(M_L) \le \log(\lambda(\phi_g)) \le \tfrac{C'}{g}$. 

The result $\ell(\mathrm{MCG}(\Sigma_g)) \asymp \tfrac{1}{g}$ by Penner  \cite{Penner91} 
tells us that 
there exists a constant $C>0$ independent of $g$ 
so that $\tfrac{1}{Cg} \le \ell_g(M_L) $. 
We conclude that 
 $\ell_g(M_L) \asymp \tfrac{1}{g}$. 
This completes the proof. 
\end{proof}

\begin{center}
\begin{figure}[t]
\includegraphics[height=4.5cm]{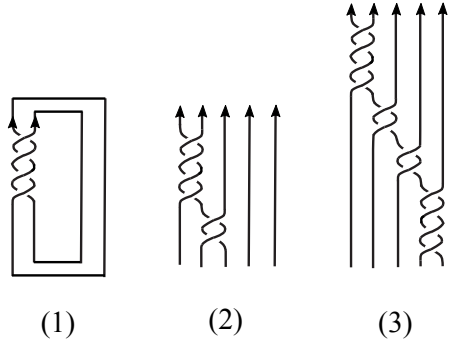}
\caption{
Case $b= \sigma_1^4 $. 
(1) $\mathrm{cl}(b)$ for $b \in B_2$.  
(2) $ \eta_{(b,1)} = \overline{b} \sigma_2^2 =  \sigma_1^4 \sigma_2^2 \in B_5$. 
(3) $\eta_{(b,1)} \cdot \mathrm{skew}(\eta_{(b,1)})= \sigma_1^4 \sigma_2^2 \sigma_3^2 \sigma_4^4 \in B_5$.}
\label{fig_kobayashi}
\end{figure}
\end{center}

\begin{thm}
\label{thm_application} 
Let  $b \in B_n$ be a  pure braid for $n \ge 2$ of the form 
$$b= \sigma_{j_1}^{2m_1} \sigma_{j_2}^{2m_2} \dots \sigma_{j_k}^{2m_k},$$
where $m_1, \dots, m_k$ are non-zero integers and 
$j_1, \dots, j_k \in \{1, \dots, n-1\}$. 
Suppose that $b$ is homogeneous, 
and each $\sigma_i$ for $i= 1, \dots, n-1$ appears in $b$ at least once, 
i.e., 
$\{j_1, \dots, j_k\}= \{1, \dots, n-1\}$. 
Then for the $2$-fold branched cover $M_{\mathrm{cl}(b)}$ of $S^3$ branched over the closure $\mathrm{cl}(b)$ of $b$,  
we have 
$\ell_g(M_{\mathrm{cl}(b)})  \asymp \dfrac{1}{g}$. 
In particular, 
if a pure braid $b \in B_n$ is of the form 
$$b =\sigma_1^{2 m_1} \sigma_2^{2 m_2} \cdots \sigma_{n-1}^{2 m_{n-1}},$$
then we have 
$\ell_g(M_{\mathrm{cl}(b)})  \asymp \dfrac{1}{g}$. 
\end{thm}

To prove Theorem \ref{thm_application}, 
we need the following lemma.

\begin{lem}
\label{lem_pseudo-Anosov}
Let  $b =   \sigma_{j_1}^{2m_1} \sigma_{j_2}^{2m_2} \dots \sigma_{j_k}^{2m_k} \in B_n$ be a  pure braid for $n \ge 2$ 
with the same assumption as in Theorem \ref{thm_application}. 
Let $\overline{b}$ be a $(2n+1)$-braid with the same braid word as $b$. 
We take a braid 
$$ \eta_{(j)} =\eta_{(b, j)}: = \overline{b} \sigma_{n}^{2 j}= 
 (\sigma_{j_1}^{2 m_1}  \sigma_{j_2}^{2 m_2}   \cdots \sigma_{j_k}^{2 m_k}) \sigma_{n}^{2 j} \in B_{2n+1}$$ 
 for a non-negative integer $j$. 
 Then $ \eta_{(j)}$ is increasing in the middle with the intersection number $u(\eta_{(j)}, n+1)= j$, and 
 the braid 
$$  \eta_{(j)} \cdot \mathrm{skew}(\eta_{(j)}) = 
(\sigma_{j_1}^{2 m_1}  \cdots \sigma_{j_k}^{2 m_{k}}) \sigma_{n}^{2 j} \cdot 
\sigma_{n+1}^{2 j}  (\sigma_{2n+1 - j_k}^{2 m_k}  \cdots \sigma_{2n+1 - j_1}^{2 m_1}) \in B_{2n+1}$$ 
 is pseudo-Anosov.  
\end{lem}

\begin{proof}
By the definition of $ \eta_{(j)}$, 
it is easy to check that 
$ \eta_{(j)}$ is increasing in the middle with $u(\eta_{(j)}, n+1)= j$. 
The braid $ \eta_{(j)} \cdot \mathrm{skew}(\eta_{(j)}) \in B_{2n+1}$ satisfies the assumption of Lemma \ref{lem_torus-link}, 
and hence it is pseudo-Anosov. 
\end{proof}

\begin{ex}
If  $b= \sigma_1^4 \in B_2$, then 
$\eta_{(1)}= \eta_{(b,1)} = \overline{b} \sigma_2^2 =  \sigma_1^4 \sigma_2^2 \in B_5$. 
By Lemma \ref{lem_pseudo-Anosov}, 
$\eta_{(1)} \cdot \mathrm{skew}(\eta_{(1)})= \sigma_1^4 \sigma_2^2 \sigma_3^2 \sigma_4^4 \in B_5$ 
is pseudo-Anosov. 
See Figure \ref{fig_kobayashi}. 
\end{ex}

Let us turn to the proof of Theorem \ref{thm_application}. 

\begin{proof}[Proof of Theorem \ref{thm_application}]
We consider the braid $\eta_{(j)} = \eta_{(b, j)} \in B_{2n+1}$ as in Lemma \ref{lem_pseudo-Anosov} 
for each $j= 1, \dots, n$. 
By Lemma \ref{lem_involution-skew}, 
$\mathrm{skew}(\eta_{(j)}) \in B_{2n+1}$ is a braid increasing in  the middle, and 
$u(\eta_{(j)}, n+1)= u(\mathrm{skew}(\eta_{(j)}), n+1) = j$. 
Note that 
$$\eta_{(j)}^{\bullet} =   \sigma_{j_1}^{2m_1} \sigma_{j_2}^{2m_2} \dots \sigma_{j_k}^{2m_k} \in B_{2n}$$ 
with the same word as the braid $b \in B_n$ 
for $j= 1, \dots, n$. 
Then we have 
$$C(\Bigl(\mathrm{skew}(\eta_{(j)})\Bigr)^{\bullet})= C(\eta_{(j)}^{\bullet})= \mathrm{cl}(b)$$
as links in $S^3$.  
By Lemma \ref{lem_pseudo-Anosov}, 
$\eta_{(j)} \cdot \mathrm{skew}(\eta_{(j)}) \in B_ {2n+1}$ is pseudo-Anosov for $j= 1, \dots, n$. 
Notice that  $\eta_{(j)} \cdot \mathrm{skew}(\eta_{(j)})$ is the skew-palindromization of $\mathrm{skew}(\eta_{(j)})$ 
(since $\mathrm{skew}: B_n \rightarrow B_n$ is an involution).  
Hence the braids $\mathrm{skew}(\eta_{(1)}),  \dots, \mathrm{skew}(\eta_{(n)}) \in B_{2n+1}$ 
satisfy the conditions (1)--(3) of Proposition \ref{prop_finite}, 
where $L =  C(\Bigl(\mathrm{skew}(\eta_{(j)})\Bigr)^{\bullet}) = \mathrm{cl}(b)$. 
Thus $\ell_g(M_{\mathrm{cl}(b)}) \asymp \dfrac{1}{g}$.  
This completes the proof. 
\end{proof}

Theorem \ref{thm_infinite} follows from the following result.

\begin{cor}
\label{corollary_lens}
 For the lens space $L_{(2m,1)}$ of type $(2m,1)$ with $m \ne 0$, 
we have $\ell_g(L_{(2m,1)}) \asymp \dfrac{1}{g}$. 
\end{cor}

\begin{proof}
The closure $\mathrm{cl}(\sigma_1^{2m} )$  of the $2$-braid $ \sigma_1^{2m} $ with $m \ne 0$ 
is the $(2m,2)$-torus link $T$. 
(See Figure \ref{fig_closure}(1).) 
The $2$-fold  branched cover $M_{T} = M_{\mathrm{cl}(\sigma_1^{2m} )}$ 
of $S^3$ branched over the $(2m,2)$-torus link $T$ 
is the lens space $L_{(2m,1)}$ of the type $(2m,1)$. 
See  \cite[p. 302]{Rolfsen90}  for example. 
This together with Theorem \ref{thm_application} completes the proof. 
\end{proof}

\begin{thm}
\label{thm_trivial-link}
Let $\sharp_{n}S^2 \times S^1$ denote the  connected sum of $n$ copies of $S^2 \times S^1$. 
For each $n \ge 1$, we have 
$\ell_g(\sharp_n S^2 \times S^1) \asymp \dfrac{1}{g}$. 
\end{thm}

\begin{center}
\begin{figure}[t]
\includegraphics[height=5.2cm]{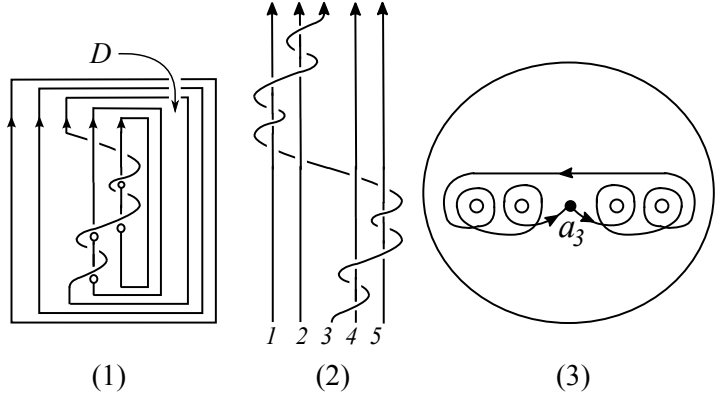}
\caption{(1) The closure $\mathrm{cl}(\beta)$ for 
$\beta= \sigma_3 \sigma_4^4 \sigma_3^3 \in B_5$ increasing in the middle 
with $u(\beta,3)=4$. ($D$ is an associated disk of the pair $(\beta, 3)$.)  
(2) The skew-palindromization $\widetilde{\beta}$ of $\beta$.  
(3) The closed curve $\gamma_{\widetilde{\beta}} $ (based at $a_3$) 
in $D_4= D^2 \setminus \{a_1, a_2, a_4, a_5\}$.}
\label{fig_Kra}
\end{figure}
\end{center}

\begin{proof}
For the $n$-component trivial link $E_n$, we have 
$M_{E_n} = \sharp_{n-1}S^2 \times S^1$. 
See  \cite[p. 300]{Rolfsen90}  for example. 
Recall that $E_n = C(e_{2n})$ for the identity element $e_{2n} \in B_{2n}$. 
To prove Theorem \ref{thm_trivial-link}, 
we check that 
$\ell_g(M_{C(e_{2n})})) \asymp \tfrac{1}{g}$.

For $n \ge 2$, we take a braid $\beta \in B_{2n+1}$ increasing in the middle 
with the following properties. 
 $$u(\beta, n+1) =  2n \equiv 0 \pmod n \ \mbox{and}\   \beta^{\bullet} =e_{2n}  \in B_{2n}.$$ 
 One can choose such a braid $\beta$ as follows. 
$$\beta = 
\sigma_{n+1} \sigma_{n+2} \cdots \sigma_{2n-1} \sigma_{2n}^4 \sigma_{2n-1}^3 \cdots \sigma_{n+2}^3 \sigma_{n+1}^3 \in B_{2n+1}.$$ 
See Figure \ref{fig_Kra}(1) for the braid 
$\beta= \sigma_3 \sigma_4^4 \sigma_3^3 \in B_5$ when $n=2$. 
Consider the skew-palindromization $\widetilde{\beta}$ of $\beta$, see Figure \ref{fig_Kra}(2). 
Since $\beta^{\bullet} =e_{2n}  \in B_{2n}$, 
it holds 
$\widetilde{\beta}^{\bullet} = e_{2n} \cdot e_{2n}= e_ {2n} \in B_{2n}$.

Let $a_1, \dots, a_{2n+1}$ be base points of $\widetilde{\beta}$. 
The projection of the $(n+1)$-th strand $\widetilde{\beta}(n+1) \subset D^2 \times [0,1]$ 
onto the first factor $D^2$ 
gives an oriented closed curve $\gamma_{\widetilde{\beta}}$ on the $2n$-punctured disk 
$D_{2n} = D^2 \setminus \{a_1, \dots, a_n, a_{n+2}, \dots, a_{2n+1}\}$,  
see Figure \ref{fig_Kra}(2)(3).  
Note that the initial point of the closed curve 
$\gamma_{\widetilde{\beta}}$ corresponds to the base point $a_{n+1}$ of $\widetilde{\beta}$. 
If we choose a braid $\beta$ increasing in the middle as above, 
then one can check that 
$\gamma_{\widetilde{\beta}}$  fills $D_ {2n}$. 
Here a closed curve $\gamma \subset D_N$ in an $N$-punctured disk $D_N$ 
 {\em fills} $D_N$ 
 if every loop that is freely homotopic to $\gamma$ intersects every essential simple closed curve in $D_N$.
By Kra's criterion \cite[Theorem 2']{Kra81}, 
$\widetilde{\beta} \in B_{2n+1}$ is pseudo-Anosov.

For each $j= 1, \dots, n$, we define a braid $\beta_{(j)} \in B_{2n+1}$ as follows: 
If $j= n$, then 
$\beta_{(n)}:= \beta$. 
If $j = 1, \dots, n-1$, then $\beta_{(j)}:= \beta \sigma_{n+1}^{2j}$. 
Notice that both braids 
$\beta \in B_{2n+1}$ and $\sigma_{n+1}^{2j} \in B_{2n+1}$ are increasing in the middles 
with the intersection numbers $2n$ and $j$ respectively. 
Lemma \ref{lem_product} tells us that 
$\beta_{(j)} = \beta \sigma_{n+1}^{2j}$ is  increasing in the middle with the intersection number 
$u(\beta_{(j)}, n+1) = 2n+ j \equiv j$ $\pmod n$. 
Moreover 
$\beta_{(j)}^{\bullet}= \beta^{\bullet} (\sigma_{n+1}^{2j})^{\bullet}= e_{2n} \cdot e_{2n}= e_{2n} \in B_{2n}$. 
Hence $E_n= C(\beta_{(j)}^{\bullet})$. 
The closed curve $\gamma_{\widetilde{\beta_{(j)}}} $ still fills $D_{2n}$. 
Hence $\widetilde{\beta_{(j)}} \in B_{2n+1}$ for $j= 1, \dots, n$ is pseudo-Anosov 
by the same criterion of Kra. 
By Proposition \ref{prop_finite},  we obtain 
$\ell_g(M_{E_n}) \asymp \frac{1}{g}$. 
This completes the proof.
\end{proof}

Finally we  prove Theorem \ref{thm_infinite-hyperbolic}. 

\begin{center}
\begin{figure}[t]
\includegraphics[height=8cm]{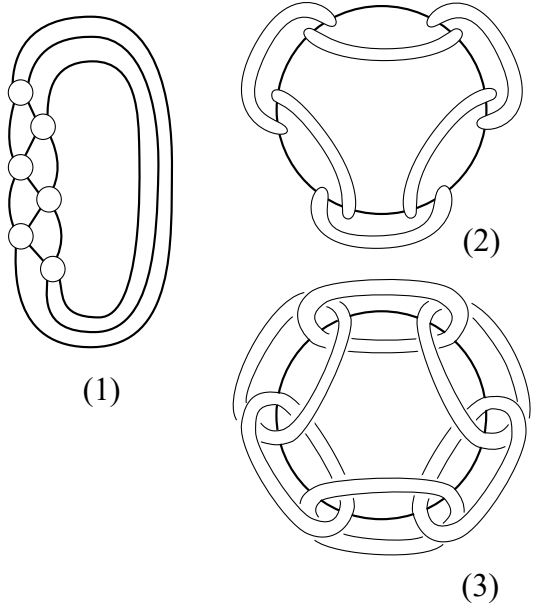}
\caption{This figure illustrates the case where $k=3$. 
(1) Let $P$ be the complement of the $2k$ $3$-balls containing twists. 
(2) We isotope (1) in order to have $2 \pi /3$ symmetry.  
(3) The $2$-fold branched cover of $P$ branched over  
$P \cap \mathrm{cl}(b_{\bm m})$ (indicated by the thick circle) 
is homeomorphic to $S^3 - \mathcal{N}(\mathcal{C}_6)$. }
\label{fig_6-chain}
\end{figure}
\end{center}

\begin{proof}[Proof of Theorem \ref{thm_infinite-hyperbolic}]
Let $b_{\bm m} \in B_3$ be a pure braid of the form 
$$
b_{\bm m} = \sigma_1^{2 m_1} \sigma_2^{2 m_2} \sigma_1^{2 m_3} \sigma_2^{2 m_4} \ldots 
\sigma_1^{2 m_{2k-1}}\sigma_2^{ 2 m_{2k}}, 
$$
where 
$k \geq 3$ and 
${\bm m} = (m_1, m_2, m_3, m_4, \ldots, m_{2k}) \in (\mathbb{Z}_{>0})^{2k}$.  
Notice that $b_{\bm m}$ is homogeneous, and 
both $\sigma_1$ and $\sigma_2$ appear in $b_{\bm m}$. 
As shown in Figure \ref{fig_6-chain}, $M_{\mathrm{cl}(b_{\bm m})}$ is a closed $3$-manifold obtained from 
the $3$-sphere $S^3$ by Dehn surgery about the minimally twisted $2k$-chain link $\mathcal{C}_{2k}$. 
Let $s_i$ $(i= 1, \dots, 2k)$ be the slope of this Dehn surgery. 
It is shown by Thurston \cite[Example 6.8.7]{thurston:notes} that 
$S^3 - \mathcal{C}_{2k}$ 
has a complete hyperbolic structure with $2k$ cusps. 
(See also \cite{Purcell,Yoshida96}.)
Hence, we have $\vol(S^3-\mathcal{C}_{2k}) > 2k v_3$, 
where $v_3= 1.01494 \dots $ 
is the volume of the regular hyperbolic tetrahedron. 
See \cite[Theorem 7]{Adams}. 
We consider an Euclidean structure on the torus boundary of 
$\mathcal{N}(\mathcal{C}_{2k})$ 
induced by a maximal disjoint horoball neighborhood about the cusps. 
Let $\lambda$ be the minimum of the Euclidean lengths of the solpes $s_i$. 
By the $2 \pi$-theorem of Gromov-Thurston \cite[Theorem 9]{BleilerHodgson96}, 
if $\lambda > 2 \pi$, then $M_{\mathrm{cl}(b_{\bm m})}$ is hyperbolic. 
Furthermore, 
from the result by Futer-Kalfagianni-Purcell \cite[Theorem 1.1]{fkp:filling}, 
we have 
$$
 \left(1-\left(\frac{2\pi}{\lambda}\right)^2\right)^{3/2} 
 \vol(S^3 - \mathcal{C}_{2k})  \leq \vol(M_{\mathrm{cl}(b_{\bm m})})  
 <  \vol(S^3 - \mathcal{C}_{2k} ).
$$  
For any $R >0$, if we choose $k$ so that $2k v_3 > R$ 
and each coordinate of ${\bm m}$ sufficiently large, 
then we have $\vol(M_{\mathrm{cl}(b_{\bm m})})>R$. 
Because the braid $b_{\bm m}$ satisfies the assumption in Theorem \ref{thm_application}, 
we see  
$\ell_g(M_{\mathrm{cl}(b_{\bm m})})  \asymp \dfrac{1}{g}$. 
\end{proof}

%
%

\bibliographystyle{alpha} 
\bibliography{skew}

\end{document}